%% file: clv-sparse-graphs-Laplacian.tex
\documentclass[11pt,reqno]{amsart}
\usepackage{amsfonts, amssymb, amsmath, enumerate, bm, xspace}
\usepackage[linktoc=page, colorlinks, linkcolor=blue, citecolor=blue]{hyperref}
\usepackage{pbox}

\usepackage{epsfig,subfigure}
\graphicspath{{figures/}}
\usepackage{epstopdf}

\numberwithin{equation}{section}

\DeclareMathOperator{\E}{\mathbb{E}}

\DeclareMathOperator*{\diag}{diag}

\renewcommand{\Pr}[2][]{\mathbb{P}_{#1} \left\{ #2 \rule{0mm}{3mm}\right\}}
\newcommand{\ip}[2]{\langle#1,#2\rangle}

\def \R {\mathbb{R}}

\def \CC {\mathcal{C}}

\def \LL {\mathcal{L}}

\def \RR {\mathcal{R}}

\def \a {\alpha}

\def \e {\varepsilon}
\def \d {\delta}
\def \l {\lambda}
\def \s {\sigma}

\def \tran {\mathsf{T}}

\def \onevector {{\bf 1}}
\def \Gr {\mathrm{G}}

\newtheorem{theorem}{Theorem}[section]
\newtheorem{proposition}[theorem]{Proposition}
\newtheorem{corollary}[theorem]{Corollary}
\newtheorem{lemma}[theorem]{Lemma}

\newtheorem{question}[theorem]{Question}

\theoremstyle{remark}
\newtheorem{remark}[theorem]{Remark}

\newtheorem{assumption}[theorem]{Assumption}

\title[]{Sparse random graphs: regularization and concentration of the Laplacian}

\author{Can M. Le \and Elizaveta Levina \and Roman Vershynin}

\email{canle@umich.edu, elevina@umich.edu, romanv@umich.edu}

\thanks{R. V. is partially supported by NSF grant 1265782 and U.S. Air Force grant FA9550-14-1-0009.  E. Levina is partially supported by NSF DMS grants 1106772 and 1159005.}

\date{\today}

\begin{document}

\begin{abstract}
We study random graphs with possibly different edge probabilities in the challenging sparse regime of bounded expected degrees. Unlike in the dense case, neither the graph adjacency matrix nor its Laplacian concentrate around their
expectations due to the highly irregular distribution of node degrees.
It has been empirically observed that simply adding a constant of order $1/n$ to each entry of the adjacency matrix  substantially improves the behavior of Laplacian. Here we prove that this regularization indeed forces Laplacian to concentrate even in sparse graphs.
As an immediate consequence in network analysis, we establish the validity of one of the simplest and fastest approaches to community detection -- regularized spectral clustering, under the stochastic block model.  Our proof of concentration
of regularized Laplacian is based on Grothendieck's inequality and factorization, combined with paving arguments.
\end{abstract}

\maketitle

\setcounter{tocdepth}{1}
\tableofcontents

\section{Introduction}\label{Sec:Introduction}
\input{Introduction}

\section{Outline of the argument}\label{Sec:outline}
\input{Outline}

\section{Community detection in sparse networks}\label{Sec:community}
\input{CommunityDetection}

\section{Grothendieck's theorem and the first core block}\label{Sec:Grothendieck}
\input{Grothendieck}

\section{Expansion of the core, and concentration of the adjacency matrix}\label{Sec:Expansion}
\input{Expansion}

\section{Decomposition of the residual}\label{Sec:DecompResidual}
\input{Residual}

\section{Concentration of the Laplacian on the core}\label{Sec:LaplacianCore}
\input{Laplacian-core}

\section{Control of the Laplacian on the residual, and proof of Theorem~\ref{thm: main}}\label{Sec:LaplacianResidual}
\input{Laplacian-residual}

\section{Proof of Corollary~\ref{Cor:CommunityDetection} (community detection)}\label{Sec:ConsistencyComDetn}
\input{ConsistencyComDetn}

\bibliography{allref,arash-sdp}
\bibliographystyle{abbrv}

\end{document}

%% file: Introduction.tex
Concentration properties of random graphs have received a substantial attention in the probability literature. In statistics, applications of these results to network analysis have been a particular focus of recent attention, discussed in more detail in Section \ref{Sec:community}.     For dense graphs (with expected node degrees growing with the number of nodes $n$), a number of results are available \cite{Oliveira2010, Rohe2011, Chaudhuri&Chung&Tsiatas2012, Qin&Rohe2013, Lei&Rinaldo2013}, mostly when the average expected node degree grows faster than $\log n$.  Real networks, however, are frequently very sparse, and no concentration results are available in the regime where the degrees are bounded by a constant. This paper makes one of the first contributions to the study of concentration in this challenging sparse regime.

\subsection{Do random graphs concentrate?}

Geometry of graphs is reflected in matrices canonically associated to them, most importantly
the adjacency and Laplacian matrices.
Concentration of random graphs can be understood as concentration of these canonical
random matrices around their means.

To recall the notion of graph Laplacian, let $A$ be the $n \times n$ adjacency matrix of an undirected finite graph $G$ on the vertex set $V$, $|V| = n$, with $A_{ij} = 1$ if there is an edge between vertices $i$ and $j$, and $0$ otherwise.
The (symmetric, normalized) {\em Laplacian} is defined as
\footnote{
We first define the Laplacian of the subgraph induced by non-isolated nodes using \eqref{eq: Laplacian}
and then extend it for the whole graph by setting the new row and column entries of isolated nodes to zero.
However, we will only work with restrictions of the Laplacian on nodes of positive degrees.
}
\begin{equation}         \label{eq: Laplacian}
\LL(A) = D^{-1/2} (D - A) D^{-1/2} = I - D^{-1/2} A D^{-1/2}.
\end{equation}
Here $I$ is the identity matrix, and $D = \diag(d_i)$ is the
diagonal matrix with degrees $d_i = \sum_{j \in V} A_{ij}$ on the diagonal.
Graph Laplacians can be thought of as discrete versions of the Laplace-Beltrami operators on
Riemannian manifolds; see \cite{ChungFan1997}.

The eigenvalues and eigenvectors of the Laplacian matrix $\LL(A)$ reflect some fundamental
geometric properties of the graph $G$. The spectrum of $\LL(A)$, which is often called the {\em graph spectrum},
is a subset of the interval $[0,2]$. The smallest eigenvalue is always zero.
The \textit{spectral gap} of $G$, which is usually defined as the minimum of the second smallest eigenvalue and the
gap between $2$ and the largest eigenvalue,
provides a quantitative measure of connectivity of $G$; see \cite{ChungFan1997}.


In this paper we will study Laplacians of {\em random} graphs.
A classical and well studied model of random graphs is the Erd\"os-R\'enyi model $G(n,p)$,
where an undirected  graph $G$ on $n$ vertices is constructed by connecting each pair of vertices independently
with a fixed probability $p$.
Although the main result of this paper is neither known nor trivial for $G(n,p)$, we shall work with
a more general, {\em inhomogeneous Erd\"os-R\'enyi model} $G(n, (p_{ij}))$
in which edges are still generated independently, but with different probabilities $p_{ij}$; see e.g. \cite{Bollobas2007}.
This includes many popular network models as special cases, including the stochastic block model \cite{Holland83}.   We ask the following basic question.



\begin{question}
\label{main_question}
  When does the Laplacian of a random graph concentrate near a deterministic matrix?
\end{question}

More precisely, for a random graph drawn from the inhomogeneous Erd\"os-R\'enyi model
$G(n, (p_{ij}))$, we are asking whether, with high probability,
$$
\|\LL(A) - \LL(\bar{A})\| \ll    \| \LL(\bar{A}) \| \quad \text{for} \quad \bar{A} = \E A = (p_{ij}).
$$
Here $\|\cdot\|$ is the operator norm and $\LL(\bar{A})$ is the Laplacian of
the weighted graph with adjacency matrix $\bar{A}$ (obtained by simply replacing $A$ with $\bar{A}$
in the definition of the Laplacian).    Since the proper scaling for $\| \LL(\bar{A}) \|$ is $\Omega(1)$ (except for a trivial graph with no edges), we can equivalently restate the question as whether
$$
\|\LL(A) - \LL(\bar{A})\| \ll   1.
$$
The answer to this question is crucial in network analysis;  see Section~\ref{Sec:community}.

\subsection{Dense graphs concentrate, sparse graphs do not}				\label{s: dense sparse}

Concentration of relatively {\em dense} random graphs -- those whose with expected degrees grow at least as fast
as $\log n$  -- is well understood.   Both the adjacency matrix and the Laplacian concentrate in this regime.
Indeed, Oliveira \cite{Oliveira2010} showed that the inhomogeneous Erd\"os-R\'enyi model satisfies
\begin{equation}         \label{eq: Laplacian dense}
\|\LL(A) - \LL(\bar{A})\| = O \Big( \sqrt{\frac{\log n}{d_0}} \Big)
\end{equation}
with high probability, where $d_0 = \min_{i \in V} \sum_{j \in V} \bar{A}_{ij}$ denotes the smallest expected degree of the graph.
The concentration inequality \eqref{eq: Laplacian dense} is non-trivial when
its right-hand side is $o(1)$, i.e.,  $d_0 \gg \log n$.

Results like \eqref{eq: Laplacian dense} for the Laplacian can be deduced from
concentration inequalities for the adjacency matrix $A$, combined with (simple)
concentration inequalities for the degrees of vertices.
Concentration for adjacency matrices can in turn be deduced either from
matrix-valued deviation inequalities (as in \cite{Oliveira2010}) or from bounds for
norms of random matrices (as in \cite{Hajek&Wu&Xu2014}).

\medskip

For {\em sparse} random graphs, with bounded expected degrees,
neither the adjacency matrix nor the Laplacian concentrate,
due to the high variance of the degree distribution (\cite{Alon&Kahale1997, Feige2005, Coja-Oghlan&Lanka2009}).
High degree vertices make the adjacency matrix unstable, and low degree vertices make the Laplacian unstable.
Indeed, a random graph in Erd\"os-R\'enyi model $G(n,p)$
has isolated vertices with high probability if the expected degree $d = np$ is $o(\log n)$.
In this case, the Laplacian $\LL(A)$ has multiple zero eigenvalues,
while $\LL(\bar{A})$ has a single eigenvalue at zero and all other eigenvalues at $1$.
This implies that $\|\LL(A) - \LL(\bar{A})\| \ge 1$, so the Laplacian fails to concentrate.
Moreover, there are vertices with degrees $\gg d$ with high probability, which
force $\|A\| \gg d$ while $\|\bar{A}\| = d$, so the adjacency matrix
does not concentrate either.

\subsection{Regularization of sparse graphs}
\label{intro:regularization}

If the concentration of sparse random graphs fails because of the degree distribution is too irregular, we may naturally ask if  {\em regularizing}
the graph in some way solves the problem.   If such a regularization is to work, it has to enforce spectrum stability and  concentration
of the Laplacian, but also preserve the graph's geometry.

One simple way to deal with isolated and very low degree nodes, proposed by \cite{amini2013pseudo} and analyzed by \cite{Joseph&Yu2013}, is to add the same small number $\tau>0$ to all entries of the adjacency matrix $A$.   That is, we replace $A$ with
\begin{equation}							\label{eq: A tau}
A_\tau := A + \tau \onevector \onevector^\tran
\end{equation}
where $\onevector$ denotes the vector in $\R^n$ whose components are all equal to $1$, and then use the Laplacian of $A_\tau$ in all subsequent analysis. This regularization creates weak edges (with weight $\tau$) between all
previously disconnected vertices, thus increasing all node degrees by $n \tau$.
Another way to deal with low degree nodes,
proposed by \cite{Chaudhuri&Chung&Tsiatas2012} and studied theoretically by \cite{Qin&Rohe2013},
is to add a constant $n\tau$ directly to the diagonal of $D$ in the definition \eqref{eq: Laplacian}.

Our paper answers the question of whether the regularization \eqref{eq: A tau} leads to concentration in the sparse case, by which we mean the case when all node degrees are bounded.
Note that for both regularizations described above,
the concentration holds trivially if we allow $\tau$ to be arbitrarily large.
The concentration was obtained if $n\tau$ grows at least as fast as $\log n$  in \cite{Qin&Rohe2013} and \cite{Joseph&Yu2013}.
However, when all expected node degrees are bounded, this requirement will lead to $\tau \onevector \onevector^\tran$ dominating $\bar{A}$.
To apply the concentration results obtained in \cite{Qin&Rohe2013, Joseph&Yu2013} to community detection,
one needs the average of expected node degrees to grow at least as $\log n$,
although the minimum expected degree can stay bounded.   This is an unavoidable consequence of using Oliveira's result \cite{Oliveira2010}, which  gives a $\log n$ factor in the bound, and makes the extension of these bounds to our case of all bounded degrees difficult.

To the best of our knowledge,  up to this point it has been unknown
whether any regularization creates informative concentration for the adjacency matrix or the graph Laplacian in the sparse case.
However,  a different Laplacian based on non-backtracking random walks was proposed in \cite{Krzakala.et.al2013spectral} and analyzed theoretically in \cite{Bordenave.et.al2015non-backtracking};  this can be thought of as an alternative and more complicated form of regularization, since introducing non-backtracking random walks also avoids isolated nodes and very low degree vertices such as those attached to the core of the graph by just one edge (which includes dangling trees).
Other methods, which are related to the non-backtracking random walks,
are the belief propagation algorithm \cite{Decelle.et.al.2011RL, Decelle.et.al.2011}
and the spectral algorithm based on the Bethe Hessian matrix \cite{Saade&Krzakala&Zdeborova2014}.
Although these methods have been empirically shown to perform well in sparse case, there is no theoretical analysis available in that regime so far.

\subsection{Sparse graphs concentrate after regularization}\label{subsec:MainResult}

We will prove that regularization \eqref{eq: A tau} does enforce concentration
of the Laplacian
even for graphs with bounded expected degrees.
To formally state our result for the inhomogeneous Erd\"os-R\'enyi model, we
shall work with random matrices of the following form.

\begin{assumption}				\label{as: A}
  $A$ is an $n\times n$ symmetric random matrix whose binary entries are jointly independent on and above the diagonal, with
  $\E A = (p_{ij})$.
  Let numbers $d \ge e$, $d_0>0$ and $\a$ be such that
  \begin{equation*}
  \max_{i,j} n p_{ij} \le d, \quad
  \min_{j}\sum_{i=1}^n p_{ij} \ge d_0, \quad
  \frac{d}{d_0} \le \a.
  \end{equation*}
\end{assumption}

\begin{theorem}[Concentration of the regularized Laplacian]				\label{thm: main}
  Let $A$ be a random matrix satisfying Assumption~\ref{as: A} and denote $\bar{A}_\tau = \E A_\tau$.
  Then for any $r\geq 1$, with probability at least $1-n^{-r}$ we have
  \begin{equation*}
    \left\|\LL(A_\tau)-\LL(\bar{A}_\tau)\right\|
    \leq C r \a^2 \log^3(d)
    \left(\frac{1}{\sqrt{d}}+\frac{1}{\sqrt{n\tau}}\right)
    \quad \text{for any } \tau>0.
  \end{equation*}
  Here $C$ denotes an absolute constant.
\end{theorem}

We will give a slightly stronger result in Theorem~\ref{thm: regularized Laplacian}.
The exponents of $r$, $\a$ and of $\log d$ are certainly not optimal,
and to keep the argument more transparent,  we did not try to optimize them.
We do not know if the $\log d$ term can be completely removed; however,  in sparse graphs
$d$ and thus $\log d$ are of constant order anyway.

\begin{remark}[Concentration around the original Laplacian]
  It is important to ask whether regularization does not destroy the original model --
  in other words, whether $\LL(\bar{A}_\tau)$ is close to $\LL(\bar{A})$.
  If we choose the regularization parameter $\tau$ so that $d \gg n\tau \gg 1$,
  it is easy to check that $\|\LL(\bar{A}_\tau) - \LL(\bar{A})\| \ll 1$, thus
  regularization almost does not affect the {\em expected} geometry of the graph.
  Together with Theorem~\ref{thm: main} this implies that
  $$
  \left\|\LL(A_\tau)-\LL(\bar{A})\right\| \ll 1.
  $$
  In other words, regularization forces the Laplacian to stay near $\LL(\bar{A})$, and this would not happen without regularization.
\end{remark}

\begin{remark}[Weighted graphs]
Since our arguments will be based on probabilistic rather than graph-theoretic considerations, the assumption that $A$
has binary entries  is not at all crucial.
With small modifications, it can be relaxed for matrices with entries that take values in the interval $[0,1]$,
and possibly for more general distributions of entries. We do not pursue such generalizations
to make the arguments more transparent.
\end{remark}

\begin{remark}[Directed graphs]
Theorem \ref{thm: main} also holds for directed graphs (whose adjacency matrices are not symmetric and have all independent entries) for a suitably modified definition of the Laplacian \eqref{eq: Laplacian},
with the two appearances of $D$ replaced by matrices of row and column degrees, respectively.
In fact, our proof starts from directed graphs and then generalizes to undirected graphs.
\end{remark}

\subsection{Concentration on the core}

As we noted in Section \ref{s: dense sparse},
sparse random graphs fail to concentrate without regularization.    We are going to show that this failure is caused by just a few vertices, $n/d$ of them.
On the rest of the vertices, which form what we call {\em the core}, both
the adjacency matrix and the Laplacian concentrate even without regularization.
The idea of constructing a graph core with large spectral gap has been exploited before.
Alon and Kahale \cite{Alon&Kahale1997} constructed a core for random 3-colorable graphs
by removing vertices with large degrees;
Feige and Ofek \cite{Feige2005} constructed a core for $G(n,p)$ in a similar way;
Coja-Oghlan and Lanka \cite{Coja-Oghlan&Lanka2009} provided a different construction
for a somewhat more general model (random graphs with given expected degrees),
which in general cannot be used to model networks with communities.
Alon and co-authors \cite{Alon&Coja-Oghlan&Han&Kang&Schacht2010} used Grothendieck's inequality
and SDP duality to construct a core;
they showed that the discrepancy of a graph, which measures how much it resembles a random graph with given expected degrees,
is determined by the spectral gap of the restriction of the Laplacian on the core (and vise versa).

The following result gives our construction of the core for the general
inhomogeneous Erd\"os-R\'enyi model $G(n,(p_{ij}))$.
As we will discuss further, our method of core construction is very different from the previous works.

\begin{theorem}[Concentration on the core]			\label{thm: core intro}
  In the setting of Theorem~\ref{thm: main}, there exists a subset $J$ of $[n]$
  which contains all but at most $n/d$ vertices, and such that
  \begin{enumerate}
    \item the adjacency matrix concentrates on $J \times J$:
      $$
      \|(A - \bar{A})_{J \times J}\| \le C r \sqrt{d} \log^3 d;
      $$
    \item the Laplacian concentrates on $J \times J$:
      $$
      \|(\LL(A) - \LL(\bar{A}))_{J \times J}\| \le \frac{C r \a^2 \log^3 d}{\sqrt{d}}.
      $$
  \end{enumerate}
\end{theorem}

We will prove this result in Theorems~\ref{thm: adjacency on core undirected} and \ref{thm: Laplacian on core} below.
Note that concentration of the  Laplacian (part 2) follows easily from concentration of the adjacency matrix (part 1).
This is because most vertices of the graph have degrees $\sim d$, so keeping only such vertices in the core
we can relate the Laplacian to the adjacency matrix as $\LL(A) \approx I - \frac{1}{d} A$.
This makes the deviation of the Laplacian in Theorem~\ref{thm: core intro}
about $d$ times smaller than the deviation of the adjacency matrix.

The rest of this paper is organized as follows.   Section \ref{Sec:outline}   outlines the steps we will take to prove the main Theorem~\ref{thm: main}.   Section \ref{Sec:community} discusses the application of this result to community detection in networks.   The proof is broken up into the following sections:  Section \ref{Sec:Grothendieck} states the Grothendieck's results we will use and applies them to the first core block (which may not yet be as large as we need).   Section \ref{Sec:Expansion} presents an expansion of the core to the required size and proves the adjacency matrix concentrates there.   Section \ref{Sec:DecompResidual} describes a decomposition of the residual of the graph (after extracting the expanded core) that will allow us to control its behavior.   Sections \ref{Sec:LaplacianCore} and \ref{Sec:LaplacianResidual} prove the result for the Laplacian, showing, respectively, that it concentrates on the core and can be controlled on the residual, which completes the proof of the main theorem.   The proof of the corollary for community detection is given in Section \ref{Sec:ConsistencyComDetn}.

%% file: Outline.tex
Our approach to proving Theorem~\ref{thm: main} consists of the following two steps.
\begin{enumerate}[\quad 1.]
  \item Remove the few ($n/d$) problematic vertices from the graph.
    On the rest of the graph -- the {\em core} -- the Laplacian concentrates even without regularization,
    by Theorem~\ref{thm: core intro}.

  \item Reattach the problematic vertices -- the {\em residual} -- back to the core, and show that regularization provides enough
    stability so that the concentration is not destroyed.
\end{enumerate}
We will address these two tasks separately.

\subsection{Construction of the core}					\label{s: core}

We start with the first step and outline the proof of the adjacency part of Theorem~\ref{thm: core intro}
(the Laplacian part follows easily, as already noted).
Our construction of the core is based on the following theorem
which combines two results due to Grothendieck, his famous inequality and a factorization theorem.
This result states that the operator norm of a matrix
can be bounded by the $\ell_\infty \to \ell_1$ norm on a large sub-matrix.
This norm is defined for an $m \times k$ matrix $B$ as
\begin{equation}\label{eq: ellinfty to ellone}
  \|B\|_{\infty \to 1} 
= \max_{x \in \{-1,1\}^m, \ y \in \{-1,1\}^k} x^\tran B y.
\end{equation}
This norm is equivalent to the {\em cut norm}, which is more frequently
used in theoretical computer science community (see \cite{Frieze&Kannan1999, Alon&Naor2006, Khot&Naor2012}).

\begin{theorem}[Grothendieck]							\label{thm: Grothendieck intro}
  For every $m \times k$ matrix $B$ and for any $\d>0$, there exists a sub-matrix
  $B_{I \times J}$ with $|I| \ge (1-\d)m$ and $|J| \ge (1-\d)k$ and such that
  $$
  \|B_{I \times J}\| \le \frac{2 \|B\|_{\infty \to 1}}{\d \sqrt{m k}}.
  $$
\end{theorem}

We will deduce and discuss this theorem in Section~\ref{s: Grothendieck theorems}.
The $\ell_\infty \to \ell_1$ norm is simpler to deal with than the operator norm, since
the maximum of the quadratic form in \eqref{eq: ellinfty to ellone} is taken with respect to vectors $x,y$
whose coordinates are all  $\pm 1$.
This can be helpful when $B$ is a random matrix. Indeed, for $B = A - \bar{A}$,
one can first use standard concentration inequalities (Bernstein's)
to control $x^\tran B y$ for fixed $x$ and $y$, and afterwards apply the union bound
over the $2^{m+k}$ possible choices of $x$, $y$.
This simple argument shows that, while concentration fails in the operator norm,
{\em adjacency matrices of sparse graphs concentrate in the $\ell_\infty \to \ell_1$ norm}:
\begin{equation}         \label{eq: cut norm random matrix intro}
\|A - \bar{A}\|_{\infty \to 1} = O(n \sqrt{d})  \quad \text{with high probability}.
\end{equation}
To see this is a concentration result, note that for large $d$ the right hand side is much smaller than
$\|\bar{A}\|_{\infty \to 1}$, which is of order $nd$.
This fact was observed in \cite{Guedon&Vershynin2014}, and we include the proof in this paper
as Lemma~\ref{lem: concentration in cut norm}.

Next, applying Grothendieck's Theorem~\ref{thm: Grothendieck intro} with $m=k=n$ and $\d=1/20$,
we obtain a subset $J_1$ which contains all but $0.1 n$ vertices, on which the
adjacency matrix concentrates:
\begin{equation}         \label{eq: first core block intro}
\|(A - \bar{A})_{J_1 \times J_1}\| = O(\sqrt{d}).
\end{equation}
Again, to understand this as concentration, note that for large $d$ the right hand side
is much smaller than $\|\bar{A}_{J_1 \times J_1}\|$, which is of order $d$.

We obtained the concentration inequality claimed in the adjacency part of Theorem~\ref{thm: core intro},
but with a core that may not be as large as we claimed. Our next goal is to reduce the number of residual vertices
from $0.1 n$ to $n/d$. To expand the core, we continue to apply the argument above to the
remainder of the matrix, thus obtaining new core blocks.
We repeat this process until the core becomes as large as required.
At the end, all the core blocks constructed this way are combined using the triangle inequality,
at the small cost of a factor polylogarithmic in $d$.

\subsection{Controlling the regularized Laplacian on the residual}

The second step is to show that regularized Laplacian $\LL(A_\tau)$
is stable  with respect to adding a few vertices.
We will quickly deduce such stability from the following sparse decomposition of the adjacency matrix.
\begin{theorem}[Sparse decomposition]		\label{thm: sparse decomposition intro}			
  In the setting of Theorem~\ref{thm: main}, we can decompose any sub-matrix $A_{I \times J}$
  with at most $n/d$ rows or columns into two matrices with disjoint support,
  $$
  A = A_\CC + A_\RR,
  $$
  in such a way that each row of $A_\RR$ and each column of $A_\CC$ will have
  at most $10 r \log r$ entries that equal $1$.
\end{theorem}

We will obtain a slightly more informative version of this result
as Theorem~\ref{thm: decomposition of residual undirected}.
The proof is not difficult.
Indeed, using a standard concentration argument it is possible to show that there exists
at least one sparse row or column of $A_{I \times J}$.
Then we can iterate the process -- remove this row or column and
find another one from the smaller sub-matrix, etc.
The removed rows and columns form the $A_\RR$ and $A_\CC$, respectively.

To use Theorem~\ref{thm: sparse decomposition intro} for our purpose,
it would be easier to drop the identity from the
definition of the Laplacian. Thus we consider the averaging operator
\begin{equation}         \label{eq: averaging operator}
L(A) := I - \LL(A) = D^{-1/2} A D^{-1/2},
\end{equation}
which is occasionally also called the Laplacian.
We show that the regularized averaging operator is small
(in the operator norm) on all sub-matrices with small dimensions.

\begin{theorem}[Residual]				\label{thm: residual intro}
  In the setting of Theorem~\ref{thm: main},
  any sub-matrix $L(A_\tau)_{I \times J}$ with at most $n/d$ rows or columns satisfies
  $$
  \|L(A_\tau)_{I \times J}\| \le \frac{2}{\sqrt{d}} + \frac{\sqrt{10 r \log d}}{\sqrt{n \tau}}
  \quad \text{for any } \tau > 0.
  $$
\end{theorem}
We will prove this result as Theorem~\ref{thm: Laplacian on residual} below.   The proof is based on the sparse decomposition constructed
in Theorem~\ref{thm: sparse decomposition intro} and proceeds as follows.
It is enough to bound the norm of $L(A_\tau)_\RR$.
By definition \eqref{eq: averaging operator}, $L(A)$ normalizes each entry of $A$ by the sums of entries in its
row and column. It is not difficult to see that Laplacians must scale accordingly, namely
\begin{equation}         \label{eq: Laplacian ratio}
\|L(A_\tau)_\RR\| \le \sqrt{\e} \, \|L(A_\tau)_{I \times J}\|
\end{equation}
if the sum of entries in each column of $L(A_\tau)_\RR$ is at most $\e$ times smaller
than the corresponding sum for $L(A_\tau)$.
Let us assume that $L(A_\tau)_{I \times J}$ has $n/d$ rows.
The sum of entries of each column of $L(A_\tau)_\RR$ is at most $n \tau/d + 10 r \log d$.
(The first term here comes from adding the regularization parameter $\tau$ to each of $n/d$ entries of the column,
and the second term comes from the sparsity of $\RR$.) The sum of entries of each column of
$L(A_\tau)_{I \times J}$ is at least $n \tau$ due to regularization.
Substituting this into \eqref{eq: Laplacian ratio}, we obtain
$$
\|L(A_\tau)_\RR\|
\le \sqrt{ \frac{n \tau/d + 10 r \log d}{n \tau} } \, \|L(A_\tau)_{I \times J}\|.
$$
Since the norm of $L$ is always bounded by $1$, this leads to the conclusion of
Theorem~\ref{thm: residual intro}.

\medskip

Finally, Theorem~\ref{thm: main} follows by combining the core part
(Theorem~\ref{thm: core intro}) with the residual part (Theorem~\ref{thm: residual intro}).
To do this, we decompose the part of the Laplacian outside the core $J \times J$
into two residual matrices, one on $J^c \times [n]$ and another on $J \times J^c$.
We use that the regularized Laplacian concentrates on the core and is small on each of the
residual matrices. Combining these bounds by triangle inequality, we obtain Theorem~\ref{thm: main}.

%% file: CommunityDetection.tex

\subsection{Stochastic models of complex networks}

Concentration results for random graphs have remarkable implications
for network analysis, specifically for understanding the behavior of
spectral clustering applied in the community detection problem.
Real world networks are often modelled as random graphs, and finding
{\em communities} -- groups of nodes that behave similarly to each other.  Most of the
models proposed for modeling communities to date are special cases of
the inhomogeneous Erd\"os-R\'enyi model, which we discussed in
Section~\ref{subsec:MainResult}.  In particular, the {\em stochastic
  block model} \cite{Holland83}  assigns one of $K$ possible community (block)
labels to each node $i$, which we will call $c_i \in \{1, \dots, K\}$,
and then assumes that the probability of an edge $p_{ij} = B_{c_i
  c_j}$, where $B$ is a symmetric $K \times K$ matrix containing the
probabilities of edges within and between communities.

For simplicity of presentation, we focus on the simplest version of
the stochastic block model, also known as the balanced planted
partition model, which assumes $K = 2$, $B_{11} = B_{22} = p$, $B_{12}
= q$, and the two communities contain the same number of nodes (we
assume that $n$ is an even number and split the set of vertices into
two equal parts $\CC_1$ and $\CC_2$).   We further assume that
$p > q$, and thus on average there are more edges
within communities than between them.  (This is a called an assortative
network model;  the disassortative case $p < q$ can in principle be treated
similarly but we will not consider it here).    We call this model of
random graphs $G(n, p, q)$.

\subsection{The community detection problem}

The community detection problem is to recover the node labels $c_i$,
$i = 1, \dots, n$
from a single realization of the random graph model, in our case $G(n,
p, q)$, or in more common notation, $G(n, \frac{a}{n}, \frac{b}{n})$.
A large literature exists on both the detection algorithms and the
theoretical results establishing when detection is possible, with the
latter mostly confined to the simplest $G(n, \frac{a}{n},
\frac{b}{n})$ model.   A conjecture was made in the physics literature
\cite{Decelle.et.al.2011} and rigorous results established in a series of papers by Mossel,
Neeman and Sly, as well as independently by two other groups -- see
\cite{Mossel&Neeman&Sly2014, Mossel&Neeman&Sly2014a,
  Mossel&Neeman&SlyOnConsistencyThresholds2014,
  Abbe&Bandeira&Hall2014, Massoulie:2014}.  It is now known that no method can do better than random guessing unless
$$
(a-b)^2 > 2 (a+b) .
$$
Further, weak consistency (fraction of mislabelled nodes going to 0 with high probability) is achievable if and only if $(a-b)^2/(a+b) \rightarrow \infty$, and strong consistency, or exact recovery (labelling {\em all} nodes correctly with high probability) requires a stronger necessary and sufficient condition given by \cite{Mossel&Neeman&SlyOnConsistencyThresholds2014} in terms of certain binomial probabilities, which is satisfied when the average expected degree $\frac{1}{2}(a+b)$ is of order $\log n$ or larger, and $a$ and $b$ are sufficiently separated.   Most existing results on community detection are obtained in the latter regime, showing exact recovery is possible when the degree grows faster than $\log n$ -- see e.g., \cite{McS01,Bickel&Chen2009}.


There are very few existing results about community detection on sparse graphs with bounded average degrees.   Consistency is no longer possible, but one can still hope to do better than random guessing above the detection threshold.  A (quite complicated) adaptive spectral algorithm by Coja-Oghlan \cite{Coj10}
achieves community detection if
$$
(a-b)^2 \ge C (a+b) \log (a+b)
$$
for a sufficiently large constant $C$.     Recently, two other spectral algorithms based on non-backtracking random walks
were proposed by Mossel, Neeman and Sly \cite{Mossel&Neeman&Sly2014}
and Massouile \cite{Massoulie:2014}, which perform detection better than random guessing (fraction of misclassified vertices is bounded away from $0.5$ as $n \to \infty$ with high probability) as long as
\begin{equation}         \label{eq: a-b a+b}
(a-b)^2 > C (a+b) \ \mbox{ for } C \ge 2.
\end{equation}
Finally, semi-definite programming approaches to community detection have been discussed and analyzed in the dense regime \cite{Chen2012, Cai2014, Amini&Levina2014}, and very recently Gu\'{e}don and Vershynin \cite{Guedon&Vershynin2014} proved that they achieve community detection in the sparse regime under
the same condition \eqref{eq: a-b a+b}, also using Grothendieck's results.

\subsection{Regularized spectral clustering in the sparse regime}

As an application of the new concentration results, we show that {\em regularized spectral clustering} \cite{Amini.et.al.2013}  can be used for community detection under the $G(n, \frac{a}{n}, \frac{b}{n})$ model in the sparse regime.   Strictly speaking, regularized spectral clustering is performed by first computing the leading $K$ eigenvectors and then applying the $K$-means clustering algorithm to estimate node labels, but since we are focusing on the case $K = 2$, we will simply show that the signs of the elements of the eigenvector corresponding to the second smallest eigenvalue
(under the $G(n, \frac{a}{n}, \frac{b}{n})$ model the eigenvector corresponding to the smallest eigenvalue 0 does not contain information about the community structure) match the partition into communities with high probability.   Passing from a concentration result on the Laplacian to a result about $K$-means clustering on its eigenvectors can be done by standard tools such as those used in \cite{Qin&Rohe2013} and is omitted here.

\begin{corollary}[Community detection in sparse graphs] \label{Cor:CommunityDetection}
Let $\e\in(0,1)$ and $r\ge 1$.	
Let $A$ be the adjacency matrix drawn from the stochastic block model $G(n,\frac{a}{n}, \frac{b}{n})$. Assume that $a>b$, $a\geq e$, and
\begin{equation}\label{eq: Gnab condition}
  (a-b)^2\ge C r^2\e^{-2}(a+b)\log^6a
\end{equation}
for some large constant $C>0$.
Choose $\tau=(d_1+\cdots+d_n)/n^2$, where $d_1,...,d_n$ are degrees of the vertices.
Denote by $v$ and $\bar{v}$ the unit-norm eigenvectors associated to the second smallest eigenvalues of $\LL(A_\tau)$ and $\LL(\bar{A}_\tau)$, respectively. Then with probability at least $1-n^{-r}$, we have
\begin{equation*}\label{Eq:ComDetecEigVecBound}
  \min_{\beta=\pm 1}\|v+\beta\bar{v}\|_2\leq \e.
\end{equation*}
In particular, the signs of the elements of $v$ correctly estimate the partition into the two communities, up to at most $\e n$ misclassified vertices.
\end{corollary}

Let us briefly explain how Corollary~\ref{Cor:CommunityDetection} follows from the new
concentration results. According to Theorem~\ref{thm: main}
and the standard perturbation results (Davis-Kahan theorem),
the eigenvectors of $\LL(A_\tau)$ approximate the corresponding eigenvectors of $\LL(\bar{A}_\tau)$ and therefore of $\LL(\bar{A})$.
The latter matrix has rank two. The trivial eigenvector of $\LL(\bar{A}_\tau)$ is $\onevector$, with all entries equal to 1. The first non-trivial eigenvector has entries $1$ and $-1$,
and it is constant on each of the two communities. Since we have a good approximation of
that eigenvector,  we can recover the communities.

\begin{remark}[Alternative regularization]

A different natural regularization \cite{Chaudhuri&Chung&Tsiatas2012, Qin&Rohe2013} we briefly mentioned in Section \ref{intro:regularization}
is to add  a constant, say $n \tau$, to the diagonal of the degree matrix $D$ in the definition of the Laplacian rather than to the adjacency matrix $A$.    Thus we have the alternative regularized Laplacian  $I - D_\tau^{-1/2} A D_\tau^{-1/2}$, where $D_\tau = D + n \tau I$.      One can think of this regularization as
{\em adding a few stars} to the graph.   Suppose for simplicity that $n \tau$ is an integer.    It is easy to check that this version of regularized Laplacian can also be obtained as follows:  add $n \tau$ new vertices, connect each of them to all existing vertices, compute the (ordinary) Laplacian of the resulting graph, and restrict it to the original $n$ vertices.    It is straightforward to show a version of Theorem~\ref{thm: main} holds for this regularization as well;  we omit it here out of space considerations.
\end{remark}

%% file: Grothendieck.tex
Our arguments will be easier to develop for non-symmetric adjacency matrices, which have all independent entries.
One can think of them as adjacency matrices of {\em directed random graphs}. So most of our analysis will
be concerning directed graphs, but in the end of some sections we will discuss undirected graphs.

We are about to start proving the adjacency part of Theorem~\ref{thm: core intro}, first for directed graphs.
Our final result will be a little stronger, see Theorems~\ref{thm: adjacency on core}
and \ref{thm: adjacency on core undirected} below,
and it will hold under the following weaker assumptions on $A$.

\begin{assumption}[Directed graph, bounded expected average degree]	\label{as: non-symmetric upper}
  $A$ is an $n\times n$ random matrix with independent binary entries,
  and $\E A = (p_{ij})$.
  Let number $d \ge e$ be such that
  $$
  \frac{1}{n} \sum_{i,j=1}^n p_{ij} \le d.
  $$
  In other words, we shall consider a directed random graph whose {\em expected average
  degree is bounded by $d$}.
\end{assumption}

In this section, we construct the first core block --
one that misses $0.1 n$ vertices and on which the adjacency matrix is concentrated
as we explained in Section~\ref{s: core}.
The construction will be based on two Grothendieck's theorems.

\subsection{Grothendieck's theorems}				\label{s: Grothendieck theorems}

Grothendieck's inequality is a fundamental result, which was originally proved in \cite{Grothendieck1953}
and formulated in \cite{Lindenstrauss1968} in the form we are going to use in this paper.
Grothendieck's inequality has found applications in many areas \cite{Alon&Coja-Oghlan&Han&Kang&Schacht2010,Pisier2012,Khot&Naor2012},
and most recently in the analysis of networks \cite{Guedon&Vershynin2014}.

\begin{theorem}[Grothendieck's inequality]	    \label{thm: Grothendieck inequality}
  Consider an $m \times k$ matrix of real numbers $B = (B_{ij})$.
  Assume that for all numbers $s_i, t_i \in \{-1,1\}$, one has
  $$
  \Big| \sum_{i,j} B_{ij} s_i t_j \Big| \le 1
  $$
  Then, for any Hilbert space $H$ and all vectors $x_i, y_i$ in $H$ with norms at most $1$, one has
  $$
  \Big| \sum_{i,j} B_{ij} \ip{x_i}{y_j} \Big| \le K_\Gr.
  $$
\end{theorem}

Here $K_\Gr$ is an absolute constant usually called {\em Grothendieck's constant}.
The best value of $K_\Gr$ is still unknown, and the best known bound is
$K_\Gr < \pi / (2 \ln(1+\sqrt{2})) \le 1.783$.

\medskip

It will be useful to formulate Grothendieck's inequality in terms of the $\ell_\infty\rightarrow\ell_1$ norm,
which is defined as
\begin{align}
\|B\|_{\infty \to 1}
&= \max_{\|t\|_\infty \le 1} \|Bt\|_1
= \max_{s \in \{-1,1\}^n, \ t \in \{-1,1\}^k} s^\tran B t \nonumber\\
&= \max_{s \in \{-1,1\}^m, \ t \in \{-1,1\}^k} \sum_{i,j} B_{ij} s_i t_j.    \label{Eq:CutNormDef}
\end{align}

Grothendieck's inequality then states that for any $m \times k$ matrix $B$,
any Hilbert space $H$ and all vectors $x_i, y_i$ in $H$ with norms at most $1$, one has
$$
\Big| \sum_{i,j} B_{ij} \ip{x_i}{y_j} \Big| \le K_\Gr \|B\|_{\infty \to 1}.
$$

\begin{remark}[Cut norm]
  The $\ell_\infty\rightarrow\ell_1$ norm is equivalent to the \textit{cut norm},
  which is often used in theoretical computer science literature
  (see \cite{Alon&Naor2006,Khot&Naor2012}), and
  which is defined as the maximal sum of entries over all sub-matrices of $B$.
  The cut norm is obtained if we allow $s_i$ and $t_j$ in \eqref{Eq:CutNormDef} to take values
  in $\{0,1\}$ as opposed to $\{-1,1\}$.
  When $A$ is the adjacency matrix of a random graph and $\bar{A} = \E A$,
  the cut-norm of $A-\bar{A}$ measures the degree of ``randomness'' of the graph,
  as it controls the {\em fluctuation of the number of edges} that run between
  any two subset of vertices.
\end{remark}

We combine Grothendieck's inequality with another result of A.~Grothendieck
(see \cite{Grothendieck1953,Pisier2012}), which characterizes the matrices $B$
for which $\sum_{i,j} B_{ij} \ip{x_i}{y_j}$ is small for all vectors $x_i, y_i$
with norms at most $1$.

\begin{theorem}[Grothendieck's factorization]					\label{thm: Grothendieck factorization}
  Consider an $m \times k$ matrix of real numbers $B = (B_{ij})$.
  Assume that for any Hilbert space $H$ and all vectors $x_i, y_i$ in $H$ with norms at most $1$,
  one has
  $$
  \Big| \sum_{i,j} B_{ij} \ip{x_i}{y_j} \Big| \le 1.
  $$
  Then there exist positive weights $\mu_i$ and $\mu'_j$ that satisfy
  $\sum_{i=1}^m \mu_i = 1$ and $\sum_{j=1}^k \mu'_j = 1$ and such that
  $$
  \big\| D_\mu^{-1/2} B D_{\mu'}^{-1/2} \big\| \leq 1,
  $$
  where $D_\mu = \diag(\mu_i)$ and $D_{\mu'} = \diag(\mu'_j)$ denote
  the diagonal matrices with the weights on the diagonal.
\end{theorem}

Combining Grothendieck's inequality and factorization, we deduce a result
that allows one to control the usual (operator) norm by the $\ell_\infty \to \ell_1$ norm
on almost all of the matrix. We already mentioned this result as Theorem~\ref{thm: Grothendieck intro}.
Let us recall it again and give a proof.

 \begin{theorem}[Grothendieck]							\label{thm: Grothendieck}
  For every $m \times k$ matrix $B$ and for any $\d>0$, there exists a sub-matrix
  $B_{I \times J}$ with $|I| \ge (1-\d)m$ and $|J| \ge (1-\d)k$ and such that
  $$
  \|B_{I \times J}\| \le \frac{2 \|B\|_{\infty \to 1}}{\d \sqrt{m k}}.
  $$
\end{theorem}

\begin{proof}
Combining Theorems~\ref{thm: Grothendieck inequality} and \ref{thm: Grothendieck factorization},
we obtain positive weights $\mu_i$ and $\mu'_j$ which sum to $1$ and satisfy
\begin{equation}         \label{eq: Grothendieck diag}
\big\| D_\mu^{-1/2} B D_{\mu'}^{-1/2} \big\| \le K_G\|B\|_{\infty \to 1}.
\end{equation}
Let us choose the set $I$ to contain the indices of the weights $\mu_i$ that are bounded
below by $(\d m)^{-1}$. Since all weights sum to one, $I$ contains at least
$(1-\d)m$ indices as required. Similarly, we define $J$ to contain the indices of the
weights $\mu_i$ that are bounded below by $(\d k)^{-1}$; this set also has the required
cardinality.

By construction, all (diagonal) entries of $D_\mu^{-1/2}$ and $D_{\mu'}^{-1/2}$
are positive and bounded above by $\sqrt{\d m}$ and $\sqrt{\d k}$ respectively.
This implies that
$$
\big\| (D_\mu^{-1/2} B D_{\mu'}^{-1/2})_{\s_1 \times \s_2} \big\|
\ge \sqrt{\delta m}\sqrt{\delta k} \|B_{\sigma_1\times\sigma_2}\|.
$$
On the other hand, by \eqref{eq: Grothendieck diag} the left hand side of this inequality
is bounded above by $K_G \|B\|_{\infty \to 1}$. This completes the proof,
since Grothendieck's constant $K_G$ is bounded by $2$.
\end{proof}

\subsection{Concentration of adjacency matrices in $\ell_\infty \to \ell_1$ norm}

As we explained in Section~\ref{s: dense sparse}, the adjacency
matrices of sparse random graphs do not concentrate {\em in the operator norm}.
Remarkably, concentration can be enforced by switching to
the $\ell_\infty \to \ell_1$ norm. We stated an informal version of this result
in \eqref{eq: cut norm random matrix intro}; now we are ready for a formal statement.
It has been proved in \cite{Guedon&Vershynin2014}; let restate and prove it here for
the reader's convenience.

\begin{lemma}[Concentration of adjacency matrices in $\ell_\infty \to \ell_1$ norm] 			\label{lem: concentration in cut norm}
  Let $A$ be a random matrix satisfying Assumption~\ref{as: non-symmetric upper}.
  Then for any $r\geq 1$ the following holds with probability at least $1-e^{-2rn}$:
  $$
  \|A-\bar{A}\|_{\infty \to 1}
  \le 5 r n \sqrt{d}.
  $$
\end{lemma}

\begin{proof}
By definition,
\begin{equation}				\label{eq: cut norm as max}
   \|A-\bar{A}\|_{\infty \to 1}
  =\max_{x,y\in\{-1,1\}^n} \sum_{i,j=1}^n(A_{i,j}-\bar{A}_{i,j})x_iy_j.
\end{equation}
For a fixed pair $x,y$, the terms $X_{ij}:=(A_{i,j}-\bar{A}_{i,j}) x_i y_j$ are independent random variables.
So we can use Bernstein's inequality (see Theorem~2.10 in \cite{Boucheron&Lugosi&Massart2013}) to control the sum
$\sum_{i,j=1}^nX_{ij}$. There are $n^2$ terms here, all of them are bounded in absolute value by one,
and their average variance is at most $d/n$.
Therefore by Bernstein's inequality, for any $t>0$ we have
\begin{eqnarray}				\label{eq: sum Xij bounded}
  \Pr{ \sum_{i,j=1}^n X_{ij} > t n^2 } \leq \exp\left(-\frac{n^2t^2/2}{d/n+t/3}\right).
\end{eqnarray}
It is easy to check that this is further bounded by $e^{-4rn}$
if we choose $t = 5 r \sqrt{d} / n$. Thus, taking a union bound over $4^n$ choices of pairs $x,y$
and using \eqref{eq: cut norm as max} and \eqref{eq: sum Xij bounded}, we obtain that
\begin{equation}\label{Eq:UnifErrCutNorm3}
\|A-\bar{A}\|_{\infty \to 1} \le t n^2 = 5 r n \sqrt{d}
\end{equation}
with probability at least $1 - 4^n \cdot e^{-4rn} \ge 1 - e^{-2rn}$.
The lemma is proved.
\end{proof}

\begin{remark}[Concentration]
  To better understand Lemma~\ref{lem: concentration in cut norm} as a concentration result,
  note that $\|\bar{A}\|_{\infty \to 1} = nd$ if $d$ is the average expected degree of the graph
  (that is, $d = \frac{1}{n} \sum_{i,j=1}^n p_{ij}$). Then the conclusion of
  Lemma~\ref{lem: concentration in cut norm}
  can then be stated as
  $$
  \|A-\bar{A}\|_{\infty \to 1}
  \le \frac{7r}{\sqrt{d}} \, \|\bar{A}\|_{\infty \to 1}.
  $$
  For large $d$, this means that $A$ concentrates near its mean in $\ell_\infty \to \ell_1$ norm.
\end{remark}

\subsection{Construction of the first core block}				\label{s: first core block}

We can now quickly deduce the existence of the first core block -- the one on which
the adjacency matrix concentrates in the operator norm, as we outlined in \eqref{eq: first core block intro}.

To do this, we first apply Lemma~\ref{lem: concentration in cut norm}, then
use Grothendieck's Theorem~\ref{thm: Grothendieck} for $m=k=n$ and $\d=1/20$,
and finally we intersect the subsets $I$ and $J$. We conclude the following.

\begin{proposition}[First core block]				\label{prop: first core block}
  Let $A$ be a matrix satisfying the conclusion of
  Concentration Lemma~\ref{lem: concentration in cut norm}.
  There exist a subset $J_1 \subseteq [n]$ which contains all but at most $0.1 n$ indices,
  and such that
  $$
  \|(A - \bar{A})_{J_1 \times J_1}\| \le C r \sqrt{d}.
  $$
\end{proposition}

\begin{remark}[Concentration]
  To better understand Lemma~\ref{lem: concentration in cut norm},
  one can check that $\|\bar{A}\| \ge d$ if $d$ is the average expected degree of the graph
  (that is, $d = \frac{1}{n} \sum_{i,j=1}^n p_{ij}$). Then the conclusion of
  Lemma~\ref{lem: concentration in cut norm}
  can then be stated as
  $$
  \|(A - \bar{A})_{J_1 \times J_1}\| \le \frac{Cr}{\sqrt{d}} \, \|\bar{A}\|.
  $$
\end{remark}

%% file: Expansion.tex
Our next goal is to expand the core so it contains all but at most $n/d$ (rather than $0.1n$) vertices.
As we explained in Section~\ref{s: core}, this will be done by repeatedly constructing
core blocks (using Grothendieck's theorems) in the parts of the matrix not yet in the core.
This time we will require a slightly stronger upper bound on the
average degrees than in Assumption~\ref{as: non-symmetric upper}.

\begin{assumption}[Directed graphs, stronger bound on expected density]		\label{as: non-symmetric upper strong}
  $A$ is an $n\times n$ random matrix with independent binary entries,
  and $\E A = (p_{ij})$.
  Let number $d \ge e$ be such that
  $$
  \max_{i,j} n p_{ij} \le d.
  $$
\end{assumption}

\subsection{Concentration in $\ell_\infty \to \ell_1$ norm on blocks}

First, we will need to sharpen  the concentration inequality of Lemma~\ref{lem: concentration in cut norm}
and make it sensitive to the size of the blocks.

\begin{lemma}[Concentration of adjacency matrices in $\ell_\infty \to \ell_1$ norm]  \label{lem: concentration on blocks in cut norm}
  Let $A$ be a random matrix satisfying Assumption~\ref{as: non-symmetric upper strong}.
  Then for any $r\geq 1$ the following holds with probability at least $1-n^{-2r}$.
  Consider a block\footnote{By block we mean a product set $I \times J$
    with arbitrary index subsets $I, J \subseteq [n]$. These subsets are not required to be intervals of successive
    integers.}
  $I \times J$ whose dimensions $m \times k$ satisfy $\min(m,k) \ge n/4d$.
  Then
  $$
  \|(A-\bar{A})_{I \times J}\|_{\infty \to 1}
  \le 30 r \sqrt{mkd}.
  $$
\end{lemma}

\begin{proof}
The proof is similar to that of Lemma~\ref{lem: concentration in cut norm}, except we
take a further union bound over the blocks $I \times J$ in the end.
Let us fix $I$ and $J$. Without loss of generality, we may assume that $m \le k$.
By definition,
\begin{equation}       \label{eq: cut norm as max on blocks}
\|(A-\bar{A})_{I\times J}\|_{\infty \to 1}
= \max_{x\in\{-1,1\}^m, \, y\in\{-1,1\}^k} \sum_{i\in I, \, j\in J}(A_{ij}-\bar{A}_{ij}) x_i y_j.
\end{equation}
For fixed pair $x,y$, we use Bernstein's inequality like in Lemma~\ref{lem: concentration in cut norm}.
Denoting $X_{ij}=(A_{i,j}-\bar{A}_{i,j})x_iy_j$, we obtain
\begin{equation}         \label{eq: sum Xij bounded on blocks}
\Pr{ \sum_{i\in I, \, j\in J}X_{i,j} > tmk } \leq \exp\left(-\frac{m k t^2/2}{d/n+t/3}\right).
\end{equation}

Deviating at this point from the proof of Lemma~\ref{lem: concentration in cut norm},
we would like this probability to be bounded by $(en/k)^{-8rk}$ in order to make room for the
later union bound over $I, J$. One can easily check that this happens if we choose
$t = 15 r \sqrt{d/mn} \log(en/k)$; this is the place where we use the assumption $m \ge n/4d$.
Thus, taking a union bound over $2^{m} \cdot 2^{k}$ choices of pairs $x,y$ and using
\eqref{eq: cut norm as max on blocks} and \eqref{eq: sum Xij bounded on blocks},
we obtain that
\begin{equation}         \label{eq: concentration on blocks better}
\|(A-\bar{A})_{I\times J}\|_{\infty \to 1} \le tkm
  = 15r\sqrt{mkd} \cdot \sqrt{\frac{k}{n}}\log\left(\frac{en}{k}\right)
\end{equation}
with probability at least $1 - 2^{m+k} \cdot (en/k)^{-8rk}$.
We continue by taking a union bound over all choices of $I$ and $J$.
Recalling our assumption that $m \le k$, we obtain that \eqref{eq: concentration on blocks better}
holds uniformly for all $I,J$, as in the statement of the lemma, with probability at least
$$
1 - \sum_{m = n/d}^n \sum_{k=m}^n\binom{n}{m}\binom{n}{k}2^{m+k}
  \left(\frac{en}{k}\right)^{-8rk}\geq 1-n^{-2r}.
$$
Thus we proved a slightly stronger version of the lemma, since
the extra term $\sqrt{\frac{k}{n}}\log\left(\frac{en}{k}\right)$ in \eqref{eq: concentration on blocks better}
is always bounded by $2$.
\end{proof}

As in Section~\ref{s: first core block}, we can combine Lemma~\ref{lem: concentration on blocks in cut norm}
with Grothendieck's Theorem~\ref{thm: Grothendieck}. We conclude
the following expansion result.

\begin{lemma}[Weak expansion of core into a block]			\label{lem: weak expansion into a block}
  Let $A$ be a matrix satisfying the conclusion of
  Concentration Lemma~\ref{lem: concentration on blocks in cut norm}.
  Consider a block $I \times J$ whose dimensions $m \times k$ satisfy
  $\min(m,k) \ge n/4d$.
  Then for every $\d \in (0,1)$ there exists a sub-block $I' \times J'$ of dimensions
  at least $(1-\d)m \times (1-\d)k$
  and such that
  $$
  \|(A - \bar{A})_{I' \times J'}\| \le C r \d^{-1} \sqrt{d}.
  $$
\end{lemma}

\subsection{Strong expansion of the core into a block}

The core sub-block $I' \times J'$ constructed in Lemma~\ref{lem: weak expansion into a block} is still too
small for our purposes. For $m \le k$, we would like $J'$ to miss the number of
columns that is a small fraction in $m$ (the smaller dimension!) rather than $k$.
To achieve this, we can apply Lemma~\ref{lem: weak expansion into a block}
repeatedly for the parts of the block not yet in the core, until
we gain the required number of columns. Let us formally state and prove this result.

\begin{proposition}[Strong expansion into a block]			\label{prop: expansion into a block}
  Let $A$ be a matrix satisfying the conclusion of
  Concentration Lemma~\ref{lem: concentration on blocks in cut norm}.
  Then any block $I \times [n]$ with $|I| =: m \ge n/4d$ rows
  contains a sub-block $I' \times J'$ of dimensions
  at least $(m-m/8) \times (n-m/8)$ and such that
  $$
  \|(A - \bar{A})_{I' \times J'}\| \le C r \sqrt{d} \log^2 d.
  $$
\end{proposition}

\begin{proof}
Let $\d \in (0,1)$ be a small parameter whose value we will chose later.
The first application of Lemma~\ref{lem: weak expansion into a block}
gives us a sub-block $I_1 \times J_1$ which misses at most $\d m$ rows
and $\d n$ columns of $I \times [n]$, and on which $A$ concentrates nicely:
$$
\|(A-\bar{A})_{I_1 \times J_1}\| \le C r \d^{-1} \sqrt{d}.
$$
If the number of missing columns is to big, i.e. $\d n > m/8$, we apply
Lemma~\ref{lem: weak expansion into a block} again for the block
consisting of the missing columns, that is for $I \times J_1^c$. It has dimensions
at least $m \times \d n$. We obtain a sub-block $I_2 \times J_2$ which misses at
most $\d m$ rows and $\d^2 n$ columns, and on which $A$ nicely concentrates:
$$
\|(A-\bar{A})_{I_2 \times J_2}\| \le C r \d^{-1} \sqrt{d}.
$$
If the number of missing columns is still too big, i.e. $\d^2 n > m/8$, we
continue this process for $I \times (J_1 \cup J_2)^c$, otherwise we stop.
Figure~\ref{fig: core-block-expansion} illustrates this process.
\begin{figure}[htp]			
  \centering \includegraphics[width=0.6\textwidth]{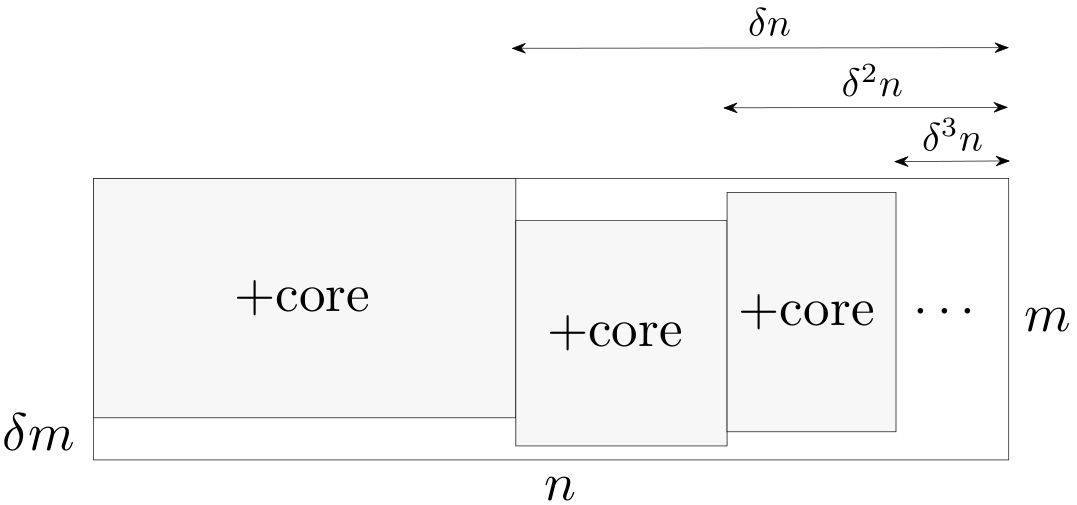}
  \caption{Core expansion into a block. First we construct the leftmost core block $I_1 \times J_1$,
    then the next core block to the right $I_2 \times J_2$, etc. The number of remaining columns
    reduces exponentially.}
   \label{fig: core-block-expansion}	
\end{figure}

The process we just described terminates after a finite number of
applications of Lemma~\ref{lem: weak expansion into a block}, which we denote by $T$.
The termination criterion yields that
\begin{equation}         \label{eq: T}
T \le \frac{\log(8n/m)}{\log(1/\d)} \le \frac{\log (8d)}{\log(1/\d)}.
\end{equation}
(The second inequality follows from the assumption that $m\ge n/d$.)
As an outcome of this process, we obtain disjoint blocks $I_t \times J_t \subseteq I \times [n]$
which satisfy
\begin{equation}         \label{eq: It Jt}
|I \setminus I_t| \le \d m \quad \text{and } \quad
|J \setminus (J_1 \cup \cdots \cup J_T)| \le m/8
\end{equation}
for all $t$.
The matrix $A$ concentrates nicely on each of these blocks:
$$
\|(A-\bar{A})_{I_t \times J_t}\| \le C r \d^{-1} \sqrt{d}.
$$

We are ready to choose the index sets $I'$ and $J'$ that would satisfy the required conclusion.
We include in $I'$ all rows of $I$ except those left out at each of the block extractions,
and we include in $J'$ all columns of each block. Formally, we define
\begin{equation}         \label{eq: I' J'}
I' := I_1 \cap \cdots \cap I_T \quad \text{and} \quad J' := J_1 \cup \cdots \cup J_T.
\end{equation}
By \eqref{eq: It Jt}, these subsets are adequately large, namely
\begin{equation}         \label{eq: I' J' sizes}
|I \setminus I'| \le T \d m \quad \text{and} \quad |J \setminus J'| \le m/8.
\end{equation}
To check that $A$ concentrates on $I' \times J'$,
we can decompose this block into (parts of) the sub-blocks we extracted before,
and use the bounds on their norms. Indeed, using \eqref{eq: I' J'} we obtain
\begin{align}
\|(A-\bar{A})_{I' \times J'}\|
&\le \sum_{t=1}^T \|(A-\bar{A})_{I' \times J_t}\|
\le \sum_{t=1}^T \|(A-\bar{A})_{I_t \times J_t}\| \nonumber\\
&\le C T r \d^{-1} \sqrt{d}.     \label{eq: conc on I' J'}
\end{align}
It remains to choose the value of $\d$. We let $\d = c/ \log(8d)$
where $c>0$ is an absolute constant. Choosing $c$ small enough
to ensure that we have $T \d \le 1/8$ according to \eqref{eq: T}. This
implies that, due to \eqref{eq: I' J' sizes}, the size of the block $I' \times J'$
we constructed is indeed at least $(m-m/8) \times (n-m/8)$ as we claimed.
Finally, using our choice of $\delta$ and the bound \eqref{eq: T} on $T$
we conclude from \eqref{eq: conc on I' J'} that
$$
\|(A-\bar{A})_{I' \times J'}\| \le C r \sqrt{d} \frac{\log^2(8d)}{c\log\left[c^{-1}\log(8d)\right]}.
$$
This is slightly better than we claimed.
\end{proof}

\subsection{Concentration of the adjacency matrix on the core: final result}

Recall that our goal is to improve upon Proposition~\ref{prop: first core block}
by expanding the core set $J_1$ until it contains all but $n/d$ vertices.
With the expansion tool given by Proposition~\ref{prop: expansion into a block},
we are one step away from this goal.
We are going to show that if the core is not yet as large
as we want, we can still expand it a bit more.

\begin{lemma}[Expansion of the core that is not too large] 			\label{lem: expansion if not too large}
  Let $A$ be a matrix satisfying the conclusion of
  Concentration Lemma~\ref{lem: concentration on blocks in cut norm}.
  Consider a subset $J$ of $[n]$ which contains all but $m \ge n/4d$ indices.
  Then there exists a subset $J'$ of $[n]$ which contains
  all but at most $m/2$ indices, and such that
  \begin{equation}         \label{eq: expansion if not too large}
  \|(A - \bar{A})_{J' \times J'}\| \le \|(A - \bar{A})_{J \times J}\| + C r \sqrt{d} \log^2 d.
  \end{equation}
\end{lemma}

\begin{proof}
We can decompose the entire $[n] \times [n]$ into three disjoint blocks --
the core block $J \times J$ and the two blocks $J^c \times [n]$ and $J \times J^c$
in which we would like to expand the core; see Figure \ref{fig: core-expansion}
for illustration.
\begin{figure}[htp]			
  \centering \includegraphics[width=0.3\textwidth]{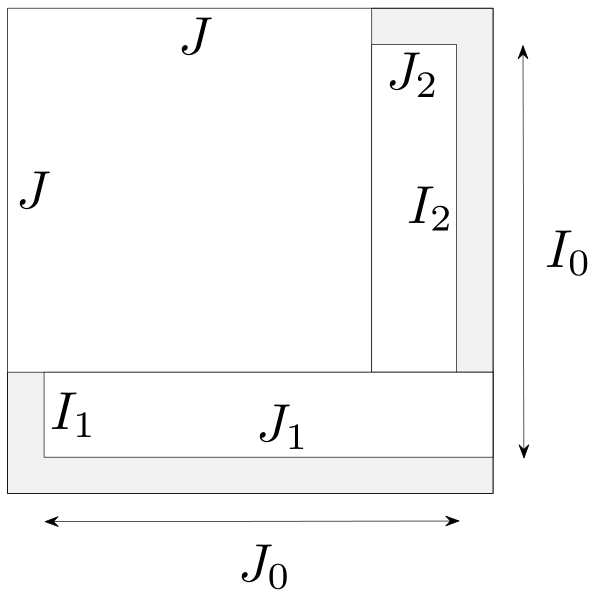}
  \caption{Expansion of the core.}
  \label{fig: core-expansion}	
\end{figure}

Applying Proposition~\ref{prop: expansion into a block} to the $m \times n$ block $J^c \times [n]$,
we obtain a sub-block $I_1 \times J_1$ which contains all but at most $m/8$ of its rows and columns,
and on which $A$ nicely concentrates:
\begin{equation}         \label{eq: conc on I1 J1}
\|(A - \bar{A})_{I_1 \times J_1}\| \le C r \sqrt{d} \log^2 d.
\end{equation}
Doing the same for the $(n-m) \times m$ block $J \times J^c$ (after transposing
and extending to an $m \times n$ block), we obtain a sub-block $I_2 \times J_2$
which again contains all but at most $m/8$ of its rows and columns,
and on which $A$ nicely concentrates:
\begin{equation}         \label{eq: conc on I2 J2}
\|(A - \bar{A})_{I_2 \times J_2}\| \le C r \sqrt{d} \log^2 d.
\end{equation}

Let $I_0$ denote the set of all rows in $[n]$ except those $m/8+m/8$ rows missed in the construction
of either of the two sub-blocks $I_1 \times J_1$ or $I_2 \times J_2$. Similarly, we let $J_0$ be the
set of the columns. The decomposition of $[n] \times [n]$ considered in the beginning of the proof
induces a decomposition of $I_0 \times J_0$ into three blocks, which are sub-blocks of
$J \times J$, $I_1 \times J_1$ and $I_2 \times J_2$.
(This follows since we remove all missing rows and columns.)
Therefore, by triangle inequality we have
$$
\|(A - \bar{A})_{I_0 \times J_0}\|
\le \|(A - \bar{A})_{J \times J}\| + \|(A - \bar{A})_{I_1 \times J_1}\| + \|(A - \bar{A})_{I_2 \times J_2}\|.
$$
Substituting \eqref{eq: conc on I1 J1} and \eqref{eq: conc on I2 J2}, we conclude
that $A$ nicely concentrates on the block $I_0 \times J_0$ -- just as we desired
in \eqref{eq: expansion if not too large}.
Since $I_0$ and $J_0$ may be different sets, we finalize the argument by choosing
$J' = I_0 \cap J_0$. Then $J' \times J'$ is a sub-block of $I_0 \times J_0$, so the
concentration inequality \eqref{eq: expansion if not too large} now holds as promised. Moreover,
since each of the sets $I_0$ and $J_0$ misses at most $m/4$ indices,
$J'$ misses at most $m/2$ indices as claimed.
\end{proof}

Lemma~\ref{lem: expansion if not too large} allows us to keep expanding the core
until it misses all but $(4d)^{-1}n$ vertices.

\begin{theorem}[Concentration of adjacency matrix on core]			\label{thm: adjacency on core}
  Let $A$ be a random matrix satisfying Assumption~\ref{as: non-symmetric upper strong}.
  Then for any $r\geq 1$ the following holds with probability at least $1-2n^{-2r}$.
  There exists a subset $J$ of $[n]$ which contains
  all but at most $n/4d$ indices, and such that
  $$
  \|(A - \bar{A})_{J \times J}\| \le C r \sqrt{d} \log^3 d.
  $$
\end{theorem}

\begin{proof}
Fix a realization of the random matrix $A$ which satisfies the conclusions of
Proposition~\ref{prop: first core block} and Concentration Lemma~\ref{lem: concentration on blocks in cut norm}.
Then Proposition~\ref{prop: first core block} gives us the first subset $J_1$ that misses at
most $0.1n$ indices, and such that
$$
\|(A - \bar{A})_{J_1 \times J_1}\| \le C r \sqrt{d}.
$$
If the number of missing indices is smaller than $n/4d$, we stop.
Otherwise we apply the Expansion Lemma~\ref{lem: expansion if not too large}.
We obtain a subset $J_2$ which misses twice fewer indices than $J_1$, and for which
$$
\|(A - \bar{A})_{J_2 \times J_2}\| \le C r \sqrt{d} + C r \sqrt{d} \log^2 d.
$$
If the new number of missing indices is smaller than $n/4d$, we stop. Otherwise
we keep applying the Expansion Lemma~\ref{lem: expansion if not too large}.

Each application of this lemma results in an additive term $C r \sqrt{d} \log^2 d$,
and it also halves the number of missing indices.
By the stopping criterion, the total number of applications is at most $\log d$.
Thus, after the process stops, the final set $J$ satisfies
$$
\|(A - \bar{A})_{J \times J}\| \le C r \sqrt{d} + C r \sqrt{d} \log^2 d \cdot \log d.
$$
This completes the proof.
\end{proof}

\subsection{Extending the result for undirected graphs}

Theorem~\ref{thm: adjacency on core} can be readily extended for undirected graphs,
where the adjacency matrix $A$ is symmetric, with only entries on and above the
diagonal that are independent.
We claimed such result in the adjacency part of Theorem~\ref{thm: core intro}; let us
restate and prove it.

\begin{theorem}[Concentration of adjacency matrix on core: undirected graphs]		\label{thm: adjacency on core undirected}
  Let $A$ be a random matrix satisfying the same requirements as in
  Assumption~\ref{as: non-symmetric upper strong}, except $A$ is symmetric.
  Then for any $r\geq 1$ the following holds with probability at least $1-2n^{-2r}$.
  There exists a subset $J$ of $[n]$ which contains
  all but at most $n/2d$ indices, and such that
  $$
  \|(A - \bar{A})_{J \times J}\| \le C r \sqrt{d} \log^3 d.
  $$
\end{theorem}

\begin{proof}
We decompose the matrix $A = A^+ + A^-$ so that each of $A^+$ and $A^-$ has all independent
entries. (Consider the parts of $A$ above and below the diagonal.)
It remains to apply Theorem~\ref{thm: adjacency on core} for $A^+$ and $A^-$ and
intersect the two subsets we obtain. The conclusion follows by triangle inequality.
\end{proof}

%% file: Residual.tex
In this section we show how to decompose the residual (in fact, any small matrix)
into two parts, one with sparse rows and the other with sparse columns.
This will lead to Theorem~\ref{thm: sparse decomposition intro},
which we will obtain in a slightly more informative form as
Theorem~\ref{thm: decomposition of residual undirected} below.

Again, we will work with directed graphs for most of the time, and in the end
discuss undirected graphs.

\subsection{Square sub-matrices: selecting a sparse row}

First we show how to select just one sparse row from square sub-matrices.
Then we extend this to rectangular matrices, and finally we iterate the process
to construct the required decomposition.

\begin{lemma}[Selecting a sparse row]				\label{lem: sparse row}
  Let $A$ be a random matrix satisfying Assumption~\ref{as: non-symmetric upper strong}.
  Then for any $r \ge 1$ the following holds with probability at least $1-n^{-2r}$.
  Every square sub-matrix of $A$ with at most $n/d$ rows
  has a row with at most $10 r \log d$ entries that equal $1$.
\end{lemma}

\begin{proof}
The argument consists of a standard application of Chernoff's inequality
and a union bound over the square sub-matrices $A_{I \times J}$.

Let us fix the dimensions $m \times m$ and the support $I \times J$ of
a sub-matrix $A_{I \times J}$ for a moment, and consider one of its rows.
The number of entries that equal $1$ in $i$-th row $S_i = \sum_{j \in I} A_{ij}$
is a sum of $m$ independent Bernoulli random variables $A_{ij}$.
Each $A_{ij}$ has expectation at most $d/n$ by the assumptions on $A$.
Thus the expected number of ones in $i$-th row is at most $1$, since
$$
\E S_i \le \frac{dm}{n} =: \mu
$$
which is bounded by $1$ by assumption on $m$.

To upgrade this to a high-probability statement, we can use Chernoff's inequality.
It implies that the probability that $i$-th row is too dense (denser than
we are seeking in the lemma) is
\begin{equation}         \label{eq: p defined}
\Pr{ S_i > 10 r \log d } \le \Big( \frac{10 r \log d}{e \mu} \Big)^{-10 r \log d} =: p.
\end{equation}
By independence, the probability that all $m$ rows of $A_{I \times J}$
are too dense is at most $p^m$.
%
%
%
%
Before we take the union bound over $I \times J$, let us simplify the
the probability $p$ in \eqref{eq: p defined}. Since $r \ge 1$, we have
$10 r \log d / (e \mu) \ge 3/\mu = 3n/(dm)$.
Therefore
\begin{equation}         \label{eq: two logs}
\log(1/p) \ge 10 r \log(d) \log \Big( \frac{3n}{dm} \Big).
\end{equation}
By assumption $m \le n/d$ on the number of rows,
both logarithms in the right hand side of \eqref{eq: two logs} are bounded below by $1$.
Then, using the elementary inequality $2ab \ge a+b$ that is valid
for all $a,b \ge 1$, we obtain
$$
\log(1/p) \ge 5 r \Big[ \log d + \log \frac{3n}{dm} \Big]
= 5 r \log \Big( \frac{3n}{m} \Big).
$$
Summarizing, we have shown that for a fixed support $I \times J$,
the probability that all $m$ rows of the sub-matrix $A_{I \times J}$
are too dense is bounded by
$$
p^m \le \Big( \frac{3n}{m} \Big)^{-5mr}.
$$

It remains to take a union bound over all supports $I \times J$.
This bounds the failure probability of the conclusion of lemma by
$$
\sum_{m=1}^{n/d} \binom{n}{m}^2 p^m
\le \sum_{m=1}^{n/d} \Big( \frac{en}{m} \Big)^{2m} \Big( \frac{3n}{m} \Big)^{-5mr}
\le n^{-2r}.
$$
This completes the proof.
\end{proof}

\subsection{Rectangular sub-matrices, and iteration}

Although we stated Lemma~\ref{lem: sparse row} for square matrices, it can be easily
adapted for rectangular matrices as well. Indeed, consider a $m \times k$ sub-matrix of $A$.
If the matrix is tall, that is $m \ge k$, then we can extend it to a square $m \times m$
sub-matrix by adding arbitrary columns from $A$. Applying Lemma~\ref{lem: sparse row},
we obtain a sparse row of the bigger matrix -- one with at most $10 r \log d$ ones in it.
Then trivially the same row of the original $m \times k$ sub-matrix will be sparse as well.

The same reasoning can be repeated for fat sub-matrices, that is for $m \le k$,
this time by applying Lemma~\ref{lem: sparse row} to the transpose of $A$.
This way we obtain a sparse column of a fat sub-matrix. Combining the two cases, we
conclude the following result that is valid for all small sub-matrices.

\begin{lemma}[Selecting a sparse row or column]				\label{lem: sparse row or column}
  Let $A$ be a random matrix satisfying Assumption~\ref{as: non-symmetric upper strong}.
  Then for any $r \ge 1$ the following holds with probability at least $1-2n^{-2r}$.
  Every sub-matrix of $A$ whose dimensions $m \times k$ satisfy $\min(m,k) \le n/d$
  has a row (if $m \ge k$) or a column (if $m \le k$) with at most $10 r \log d$ entries
  that equal $1$.
\end{lemma}

Iterating this result -- selecting rows and columns one by one -- we are going to
obtain a desired decomposition of the residual.
Here we adopt the following convention. Given a subset $\RR$ of $[n] \times [n]$,
we denote by $A_\RR$ the $n \times n$ matrix\footnote{This does not exactly agree with
  our usage of $A_{I \times J}$ which denotes an $|I| \times |J|$ matrix, but this slight
  disagreement will not cause confusion.} that has the same entries as $A$ on $\CC$
and zero outside $\RR$.

\begin{theorem}[Decomposition of the residual]					\label{thm: decomposition of residual}
  Let $A$ be a random matrix satisfying Assumption~\ref{as: non-symmetric upper strong}.
  Then for any $r \ge 1$ the following holds with probability $1-2n^{-2r}$.
  Every index subset $I \times J$ of $[n] \times [n]$ whose dimensions $m \times k$ satisfy
  $\min(m,k) \le n/d$
  can be decomposed into two disjoint subsets $\RR$ and $\CC$
  with the following properties:
    \begin{enumerate}[(i)]
      \item each row of $\RR$ and each column of $\CC$
        have at most $\min(m,k)$ entries;\footnote{Formally, for $\RR$ this means that
        $|\{j:\; (i,j) \in \RR\}| \le \min(m,k)$
        for each $i \in [n]$, and similarly for $\CC$.}
      \item each row of the matrix $A_\RR$ and each column of the matrix $A_\CC$
        have at most $10 r \log d$ entries that equal $1$.
    \end{enumerate}
\end{theorem}

\begin{proof}
Let us fix a realization of $A$ for which the conclusion of Lemma~\ref{lem: sparse row or column}
holds. Suppose we would like to decompose an $m \times k$ sub-matrix $A_{I \times J}$.
According to Lemma~\ref{lem: sparse row or column}, it has a sparse row or column.
Remove this row or column, and apply Lemma~\ref{lem: sparse row or column}
for the remaining sub-matrix. We obtain a sparse row or column of the smaller matrix.
Remove it as well, and apply Lemma~\ref{lem: sparse row or column} for the remaining sub-matrix.
Continue this process until we removed everything from $A_{I \times J}$.
Then define $\RR$ to be the union of all rows we removed throughout this process,
and $\CC$ the union of the removed columns. By construction, $\RR$ and $\CC$ satisfy
part (ii) of the conclusion.

Part (i) follows by analyzing the construction of $\RR$ and $\CC$.
Without loss of generality, let $m \le k$.
The construction starts by removing columns
of $A_{I \times J}$ (which obviously have $m$ entries as required) until the aspect ratio
reverses, i.e. there remain fewer columns than $m$. After that point, both dimensions
of the remaining sub-matrix are again bounded by $m$, so part (i) follows.
\end{proof}

\subsection{Extending the result for undirected graphs}

Theorem~\ref{thm: decomposition of residual} can be readily extended for undirected graphs.
We stated such result as Theorem~\ref{thm: residual intro}; let us restate it in
a somewhat more informative form.

\begin{theorem}[Decomposition of the residual, undirected graphs]					\label{thm: decomposition of residual undirected}
  Let $A$ be a random matrix satisfying the same requirements as in
  Assumption~\ref{as: non-symmetric upper strong}, except $A$ is symmetric.
  Then for any $r \ge 1$ the following holds with probability $1-2n^{-2r}$.
  Every index subset $I \times J$ of $[n] \times [n]$ whose dimensions $m \times k$ satisfy
  $\min(m,k) \le n/d$
  can be decomposed into two disjoint subsets $\RR$ and $\CC$
  with the following properties:
    \begin{enumerate}[(i)]
      \item each row of $\RR$ and each column of $\CC$
        have at most $2\min(m,k)$ entries;
      \item each row of the matrix $A_\RR$ and each column of the matrix $A_\CC$
        have at most $20 r \log d$ entries that equal $1$.
    \end{enumerate}
\end{theorem}

\begin{proof}
We decompose the matrix $A = A^+ + A^-$ so that each of $A^+$ and $A^-$ has all independent
entries. (Consider the parts of $A$ above and below the diagonal.)
It remains to apply Theorem~\ref{thm: adjacency on core} for $A^+$ and $A^-$ and
choose $\RR$ to be the union of the disjoint sets $\RR^+$ and $\RR^-$ we obtain this way;
similarly for $\CC$. The conclusion follows trivially.
\end{proof}

%% file: Laplacian-core.tex
In this section we translate the concentration result on the core,
Theorems~\ref{thm: adjacency on core}, from adjacency matrices to Laplacian matrices.
This will lead to the second part of Theorem~\ref{thm: core intro}.

From now on, we will focus on undirected graphs, where $A$ is a symmetric matrix.
Throughout this section, it will be more convenient to work with the alternative Laplacian
defined in \eqref{eq: averaging operator} as
$$
L(A) = I - \LL(A) = D^{-1/2} A D^{-1/2}.
$$
Clearly, the concentration results are the same for both definitions of Laplacian, since
$L(A) - L(\bar{A}) = \LL(A) - \LL(\bar{A})$ (and similarly for $A_\tau$).

\subsection{Concentration of degrees}

We will easily deduce concentration of $L(A)$ on the core from concentration
of adjacency matrix $A$ (which we already proved in Theorems~\ref{thm: adjacency on core})
and the degree matrix $D = \diag(d_j)$.
The following lemma establishes concentration of $D$ on the core.

\begin{lemma}[Concentration of degrees on core]					\label{lem: degrees on core}
  Let $A$ be a random matrix satisfying the same requirements as
  in Assumption~\ref{as: non-symmetric upper strong}, except $A$ is symmetric.
  Then for any $r \ge 1$, the following holds with probability at least $1-n^{-2r}$.
  There exists a subset $J$ of $[n]$ which contains
  all but at most $n/2d$ indices, and such that the degrees $d_j = \sum_{i=1}^n A_{ij}$
  satisfy
  $$
  |d_j - \E d_j| \le 30 r \sqrt{d \log d} \quad \text{for all } j \in J.
  $$
\end{lemma}

\begin{proof}
Let us fix $j \in [n]$ for a moment. We decompose $A$ into an upper triangular and a lower
triangular matrix, each of which has independent entries. This induces the decomposition of
the degrees
$$d_j = \sum_{i=1}^n A_{ij} = \sum_{i=j}^n A_{ij} + \sum_{i=1}^{j-1} A_{ij} =:  d_j^- + d_j^+.$$
By triangle inequality, it is enough to show that $d_j^-$ and $d_j^+$ concentrate near
their own expected values. Without loss of generality, let us do this for $d_j^+$.

By construction, $d_j^+$ is a sum of $n$ independent Bernoulli random variables
(including $n-j+1$ zeros)
whose variances are all bounded by $d/n$ by assumption on $A$.
Thus Bernstein's inequality (see Theorem~2.10 in \cite{Boucheron&Lugosi&Massart2013}) yields
$$
\Pr{ |d_j^+ - \E d_j^+| > nt } \le \exp \Big( - \frac{n t^2/2}{d/n+t/3} \Big), \quad t > 0.
$$
Choosing $t = 15 (r/n) \sqrt{d \log d}$ and simplifying the probability bound, we obtain
$$
\Pr{ |d_j^+ - \E d_j^+| > 15 r \sqrt{d \log d} } \le d^{-13r}.
$$

We choose $J^+$ to consist of the indices for which $|d_j^+ - \E d_j^+| \le 15 r \sqrt{d \log d}$.
To control the size of the complement $(J^+)^c$, we may view it as a sum of $n$ independent
Bernoulli random variables, each with expectation at most $d^{-13r}$.
Thus $\E|(J^+)^c| \le n d^{-13r}$, and Chernoff's inequality implies that
$$
\Pr{ |(J^+)^c| > n/4d } \le \Big( \frac{n/4d}{e \cdot n d^{-13r}} \Big)^{-n/4d}.
$$
Simplifying, we see that this probability is bounded by $2n^{-3r}$.

Repeating the argument for $d_j^-$, we obtain a similar set $J^-$.
Choosing $J$ to be the intersection of $J^+$ and $J^-$ and combining the two
concentration bounds by triangle inequality, we complete the proof.
\end{proof}

\subsection{Concentration of Laplacian on core}

We are ready to prove the second part of Theorem~\ref{thm: core intro}, which we
restate as follows.

\begin{theorem}[Concentration of Laplacian on core]	\label{thm: Laplacian on core}
  Let $A$ be a matrix satisfying Assumption~\ref{as: A}.
  Then for any $r \ge 1$, the following holds with probability at least $1-3n^{-2r}$.
  There exists a subset $J$ of $[n]$ which contains
  all but at most $n/d$ indices, and such that
  \begin{equation}         \label{eq: Laplacian on core}
  \|(L(A) - L(\bar{A}))_{J \times J}\| \le \frac{C r \a^2 \log^3 d}{\sqrt{d}}.
  \end{equation}
\end{theorem}

\begin{proof}
We need to compare the Laplacians
$$
L(A) = D^{-1/2} A D^{1/2} \quad \text{and} \quad L(\bar{A}) = \bar{D}^{-1/2} \bar{A} \bar{D}^{-1/2}
$$
on a big core block $J \times J$,
where $D = \diag(d_i)$ contains the actual degrees $d_i$, and
$\bar{D} = \diag(\bar{d_i})$ the expected degrees $\bar{d_i} = \E d_i$.

We get the core set $J$ by intersecting the two corresponding sets on which
$A$ concentrates (from Theorem~\ref{thm: adjacency on core undirected})
and the degree matrix $D$ concentrates (from Lemma~\ref{lem: degrees on core}).
To keep the notation simple, let us write the Laplacians on the core as
$$
L(A)_{J \times J} = B R B \quad \text{and} \quad L(\bar{A})_{J \times J} = \bar{B} \bar{R} \bar{B},
$$
where obviously $R = A_{J \times J}$,  $\bar{R} = \bar{A}_{J \times J}$,
$B = D^{-1/2}_{J \times J}$ and $\bar{B} = \bar{D}^{-1/2}_{J \times J}$.
Then we can express the difference of the Laplacians as a telescoping sum
\begin{equation}         \label{eq: telescoping sum}
(L(A) - L(\bar{A}))_{J \times J} = B(R-\bar{R})B + B\bar{R}(B-\bar{B}) + (B-\bar{B})\bar{R}\bar{B}.
\end{equation}
We will estimate each of the three terms separately.

By the conclusion of Theorem~\ref{thm: adjacency on core}, we have
\begin{equation}         \label{eq: R Rbar}
\|R - \bar{R}\| \le C r \sqrt{d} \log^3 d.
\end{equation}
Moreover, since all entries of $\bar{A}$ are bounded by $d/n$ by assumption,
we have $\|\bar{A}\| \le d$, and in particular the sub-matrix $\bar{R}$ must also satisfy
\begin{equation}         \label{eq: Rbar}
\|\bar{R}\| \le d.
\end{equation}

Next we compare $B$ and $\bar{B}$, which are diagonal matrices with entries
$d_j^{-1/2}$ and $\bar{d_j}^{-1/2}$ on the diagonal, respectively.
Since $\bar{d_j} \ge d_0$ by assumption, we have
\begin{equation}         \label{eq: Bbar}
\|\bar{B}\| \le \frac{1}{\sqrt{d_0}}.
\end{equation}
Moreover, by the conclusion of Lemma~\ref{lem: degrees on core}, the degrees satisfy
\begin{equation}         \label{eq: d dbar}
|d_j - \bar{d_j}| \le 30 r \sqrt{d \log d} \quad \text{for all } j \in J.
\end{equation}
We can assume that the right hand side here is bounded by $d_0/2$;
otherwise the right hand side in the desired bound \eqref{eq: Laplacian on core}
is greater than two, which makes the bound trivially true.
Therefore, in particular, \eqref{eq: d dbar} implies
\begin{equation}         \label{eq: dj large}
d_j \ge \bar{d_j} - d_0/2 \ge d_0/2.
\end{equation}
The difference between the corresponding entries
of $B$ and $\bar{B}$ is
$$
\big|d_j^{-1/2} - \bar{d_j}^{-1/2}\big|
= \frac{|d_j - \bar{d_j}|}{(d_j^{1/2} + \bar{d_j}^{1/2})(d_j \bar{d_j})^{1/2}}.
$$
Since $d_j \ge d_0$ by definition, $\bar{d_j} \ge d_0/2$ by \eqref{eq: dj large},
and $|d_j - \bar{d_j}|$ is small by \eqref{eq: d dbar}, this expression
is bounded by
$$
\frac{30 r \sqrt{d \log d}}{d_0^{3/2}} = \frac{30 r \sqrt{\alpha \log d}}{d_0}.
$$
This and \eqref{eq: dj large} implies that
\begin{equation}         \label{eq: B Bbar}
\|B - \bar{B}\| \le \frac{30 r \sqrt{\alpha \log d}}{d_0} \quad \text{and} \quad
\|B\| \le \frac{2}{\sqrt{d_0}}.
\end{equation}

It remains to substitute into \eqref{eq: telescoping sum} the bounds \eqref{eq: R Rbar}
for $R-\bar{R}$, \eqref{eq: Rbar} for $\bar{R}$, \eqref{eq: B Bbar} for $B-\bar{B}$ and $B$,
and \eqref{eq: Bbar} for $\bar{B}$.
Using triangle inequality and recalling that $d=\a d_0$,
we obtain \eqref{eq: Laplacian on core} and complete the proof.
\end{proof}

\begin{remark}[Regularized Laplacian]
    We just showed that the Laplacian concentrates on the core even without regularization.
    It is also true with regularization.
    Indeed, Theorem~\ref{thm: Laplacian on core} holds for the regularized Laplacian
    $L(A_\tau) = I - \LL(A_\tau)$, and they state that
    \begin{equation}         \label{eq: regularized Laplacian on core}
    \|(L(A_\tau) - L(\bar{A}_\tau))_{J \times J}\| \le \frac{C r \a^2 \log^3 d}{\sqrt{d}} \quad \text{for any } \tau \ge 0.
    \end{equation}
    This is true because Theorem~\ref{thm: Laplacian on core} is based on
    concentration of the adjacency matrix $A$ and the degree matrix $D$ on the core.
    Both of these results trivially hold with regularization as well as without it,
    as the regularization $\tau$ parameter cancels out, e.g. $A_\tau - \bar{A}_\tau = A-\bar{A}$.
    We leave details to the interested reader.
\end{remark}

%% file: Laplacian-residual.tex
\subsection{Laplacian is small on the residual}

Now we demonstrate how regularization makes Laplacian more stable.
We express this as the fact that small sub-matrices of the regularized Laplacian $L(A_\tau)$
have small norms. This fact can be easily deduced from the sparse decomposition of such matrices
that we constructed in Theorem~\ref{thm: decomposition of residual} and the following
elementary observation.

\begin{lemma}[Restriction of Laplacian]						\label{lem: restriction Laplacian}
  Let $B$ be an $n \times n$ symmetric matrix with non-negative entries,
  and let $\CC$ be a subset of $[n] \times [n]$.
  Consider the $n \times n$ matrix $B_\CC$ that has the same entries as $B$ on $\CC$ and
  zero outside $\CC$.
  Let $\e \in (0,1)$.
  Suppose the sum of entries of each row of $B_\CC$ is at most $\e$ times
  the sum of entries of the corresponding row of $B$.
  Then
  $$
  \|(L(B))_\CC\| \le \sqrt{\e}.
  $$
\end{lemma}

\begin{proof}
Let us denote by $\tilde{L}(B_\CC)$ an analog of the Laplacian for possibly non-symmetric matrix $B_\CC$, that is
\begin{equation*}
  \tilde{L}(B_\CC) = D_r^{-1/2} B_\CC D_c^{-1/2}.
\end{equation*}
Here $D_r=\mathrm{diag}(B_\CC\onevector)$ is a diagonal matrix and each diagonal
entry $(D_r)_{i,i}$ of $D_r$ is the sum of entries of $i$-th \textit{row} of $B_\CC$;
$D_c=\mathrm{diag}(B_\CC^T\onevector)$ is a diagonal matrix and $(D_c)_{i,i}$
is the sum of entries of $i$-th \textit{column} of $B_\CC$.
Note that we can write $L(B)_\CC$ as
$$L(B)_\CC=D^{-1/2}B_\CC D^{-1/2},$$
where $D=\mathrm{diag}(B\onevector)=\mathrm{diag}(B^T\onevector)$.
We have $(D_r)_{i,i}\le \e D_{i,i}$ by the assumption and $(D_c)_{i,i}\leq D_{i,i}$
because $\CC$ is a subset of $[n]\times[n]$. Since entries of both $\tilde{L}(B_\CC)$
and $(L(B))_\CC$ are non-negative, we obtain
$$\|(L(B))_{\CC}\|\leq \sqrt{\e}\|\tilde{L}(B_\CC)\|.$$
It remains to prove $\|\tilde{L}(B_\CC)\|\leq 1$. To see this, consider an $2n\times 2n$ symmetric matrix
\begin{equation*}
  S=\left(
      \begin{array}{cc}
        0_n & B_\CC \\
        B_\CC^T & 0_n \\
      \end{array}
    \right),
\end{equation*}
where $0_n$ is an $n\times n$ matrix whose entries are zero. The Laplacian of $S$ has the form
\begin{equation*}
  L(S)=\left(
      \begin{array}{cc}
        0_n & \tilde{L}(B_\CC) \\
        \tilde{L}(B_\CC)^T & 0_n \\
      \end{array}
    \right).
\end{equation*}
Since $L(S)$ has norm one, it follows that $\|\tilde{L}(B_\CC)\|\leq 1$. This completes the proof.
\end{proof}

\begin{theorem}[Regularized Laplacian on residual]		\label{thm: Laplacian on residual}
  Let $A$ be a random matrix satisfying the same requirements as in
  Assumption~\ref{as: non-symmetric upper strong}, except $A$ is symmetric.
  Then for any $r \ge 1$ the following holds with probability $1-2n^{-2r}$.
  Any sub-matrix $L(A_\tau)_{I \times J}$ of the regularized Laplacian $L(A_\tau)$
  with at most $n/d$ rows or columns satisfies
  $$
  \|L(A_\tau)_{I \times J}\| \le \frac{2}{\sqrt{d}} + \frac{\sqrt{40 r \log d}}{\sqrt{n \tau}}
  \quad \text{for any } \tau > 0.
  $$
\end{theorem}

\begin{proof}
The decomposition $I \times J = \RR \cup \CC$ we constructed
in Theorem~\ref{thm: decomposition of residual undirected} reduces the problem
to bounding $L(A_\tau)_\RR$ and $L(A_\tau)_\CC$.
Let us focus on $L(A_\tau)_\RR$.
Recall that every row of the index set $\RR$ has at most $n/d$ entries, and
every row of the matrix $A_\RR$ has at most $10 r \log d$ entries that equal one
(while all other entries are zero). This implies that
the sum of entries of each row of $(A_\tau)_\RR$ is bounded by
$$
\frac{n \tau}{d} + 10 r \log d.
$$
We compare this to the sum of the entries of each row of $A_\tau$, which is trivially
at least $n \tau$. It is worthwhile to note that this is the only place in the entire argument
where regularization is crucially used.
Applying the Restriction Lemma~\ref{lem: restriction Laplacian},
we obtain
$$
\|L(A_\tau)_\RR\| \le \sqrt{ \frac{1}{d} + \frac{10 r \log d}{n \tau} }.
$$
Repeating the same reasoning for columns, we obtain the same bound for $L(A_\tau)_\CC$.
Using triangle inequality and simplifying the expression, we conclude the desired bound
for $L(A_\tau)_{I \times J}$.
\end{proof}

Let us notice a similar, and much simpler, bound for the Laplacian of the regularized expected matrix
$\bar{A}_\tau = \E A_\tau$.

\begin{lemma}[Regularized Laplacian of the expected matrix on the residual]		\label{lem: Laplacian expected on residual}
  Let $A$ be a random matrix satisfying the same requirements as in
  Assumption~\ref{as: non-symmetric upper strong}, except $A$ is symmetric.
  Then any sub-matrix $L(\bar{A}_\tau)_{I \times J}$ with at most $n/d$ rows or columns satisfies
  $$
  \|L(\bar{A}_\tau)_{I \times J}\| \le \frac{2}{\sqrt{d}} + \frac{2}{\sqrt{n \tau}}
  \quad \text{for any } \tau > 0.
  $$
\end{lemma}

\begin{proof}
Assume that $L(\bar{A}_\tau)_{I \times J}$ has at most $n/d$ rows.
Recall that the matrix $\bar{A}_\tau$ has entries $n p_{ij} + \tau$.
Then the sum of entries of each column of the sub-matrix $(\bar{A}_\tau)_{I \times J}$ is at most
$$
\frac{n}{d} \cdot \max_j (p_{ij} + \tau) \le \frac{n \tau}{d} + 1.
$$
We compare this to the sum of entries of each column of $A_\tau$, which is at least $n \tau$.
Applying the Restriction Lemma~\ref{lem: restriction Laplacian}, we obtain
$$
\|L(\bar{A}_\tau)_{I \times J}\| \le \sqrt{ \frac{1}{d} + \frac{1}{n\tau} }.
$$
This leads to the desired conclusion.
\end{proof}

\subsection{Concentration of the regularized Laplacian}

We are ready to deduce the main Theorem~\ref{thm: main} in a slightly stronger form.

\begin{theorem}[Concentration of the regularized Laplacian]				\label{thm: regularized Laplacian}
  Let $A$ be a random matrix satisfying Assumption~\ref{as: A}.
  Then for any $r\geq 1$, with probability at least $1-n^{-r}$ we have
  $$
  \|L(A_\tau)-L(\bar{A}_\tau) \|
  \le \frac{C r \a^2 \log^3 d}{\sqrt{d}} + \frac{20 \sqrt{r \log d}}{\sqrt{n \tau}} \quad \text{for any } \tau > 0.
  $$
\end{theorem}

\begin{proof}
The proof is a combination of the Concentration Theorem~\ref{thm: Laplacian on core}
and the Restriction Theorem~\ref{thm: Laplacian on residual}.
Fix a realization of $A$ for which the conclusions of both of these results hold.
Theorem~\ref{thm: Laplacian on core} yields the existence of
a core set $J$ that contains all but at most $n/d$ indices from $[n]$,
and on which the regularized Laplacian concentrates:
\begin{equation}         \label{eq: Laplacian concentrates conclusion}
\|(L(A_\tau) - L(\bar{A}_\tau))_{J \times J}\| \le \frac{C r \a^2 \log^3 d}{\sqrt{d}}.
\end{equation}
(Here we used the version \eqref{eq: regularized Laplacian on core} that is valid
for the regularized Laplacian.)

Next, let us decompose the residual $[n] \times [n] \setminus J \times J$ into
two blocks $J^c \times [n]$ and $J \times J^c$. The first block has at most $n/d$ rows,
so the conclusion of Restriction Theorem~\ref{thm: Laplacian on residual} applies to it.
It follows that
$$
\|L(A_\tau)_{J^c \times [n]}\| \le
\frac{2}{\sqrt{d}} + \frac{\sqrt{40 r \log d}}{\sqrt{n \tau}},
$$
An even simpler bound holds for the expected version $L(\bar{A}_\tau)_{J^c \times [n]}$
according to Lemma~\ref{lem: Laplacian expected on residual}.
Summing these two bounds by triangle inequality, we conclude that
that
$$
\|(L(A_\tau) - L(\bar{A}_\tau))_{J^c \times [n]}\|
\le \frac{4}{\sqrt{d}} + \frac{10\sqrt{r \log d}}{\sqrt{n \tau}}.
$$
In a similar way we obtain the same
bound for the restriction onto the second residual block, $J \times J^c$.
Combining these two bounds with \eqref{eq: Laplacian
concentrates conclusion}, we complete the proof by triangle inequality.
\end{proof}

%% file: ConsistencyComDetn.tex
\begin{proof}[Proof of Corollary~\ref{Cor:CommunityDetection}]
Note that $n\tau$ is the average node degree with expectation $(a+b)/2$. Using Bernstein's inequality (see Theorem~2.10 in ~\cite{Boucheron&Lugosi&Massart2013}), it is easy to check that
with probability at least $1-e^{-rn}$, we have
\begin{equation}\label{Eq:AverDegrCon0}
  \Big|n\tau-\frac{a+b}{2}\Big|\leq \frac{16r}{\sqrt{a+b}}\frac{a+b}{2}.
\end{equation}
It follows from assumption \eqref{eq: Gnab condition} (and increasing the constant $C$ if necessary) that $16r/\sqrt{a+b}\le 1/2$. Therefore \eqref{Eq:AverDegrCon0} implies
\begin{equation}\label{Eq:AverDegrCon}
  \Big|n\tau-\frac{a+b}{2}\Big|\le\frac{a+b}{4}.
\end{equation}
Let us fix a realization of the random matrix $A$ which satisfies \eqref{Eq:AverDegrCon} and the conclusion of Theorem~\ref{thm: main}.
For the model $G(n,\frac{a}{n},\frac{b}{n})$ we have $d=a$ and $\alpha\le 2$. From Theorem~\ref{thm: main} and \eqref{Eq:AverDegrCon} we obtain
\begin{eqnarray}\label{Eq:TwoComLapBound}
  \left\|L(A_\tau)-L(\bar{A}_\tau)\right\|
  &<& C^\prime r log^3(a)\left(\frac{1}{\sqrt{a}}+\frac{1}{\sqrt{n\tau}}\right) \\
  \nonumber&\le& \frac{3C^\prime r\log^3a}{\sqrt{a}}=:\d,
\end{eqnarray}
for some absolute constant $C^\prime>0$.

We will use Davis-Kahan Theorem (see Theorem VII.3.2 in \cite{Bhatia1996}) and \eqref{Eq:TwoComLapBound}
to bound the difference between $v$ and $\bar{v}$.
Matrix $L(\bar{A}_\tau)$ has two non-zero eigenvalues: $\lambda_1=1$ and $\lambda_2=(a-b)/(a+b+n\tau)$.
By \eqref{Eq:AverDegrCon} we have
\begin{equation}\label{Eq:EigValBound}
  \frac{4(a-b)}{7(a+b)}\leq\lambda_2\leq\frac{4(a-b)}{5(a+b)}\leq\frac{4}{5}.
\end{equation}
To upper-bound the gaps in the spectra of $L(A_\tau)$ and $L(\bar{A}_\tau)$, let us denote
$$S=(\lambda_2-\d,4/5+\d), \quad S^\prime=(-\d,\d)\cup(1-\d,1+\d).$$
Then $\lambda_2\in S$ because $\l_2\le 4/5$ by \eqref{Eq:EigValBound}; the remaining eigenvalues of $L(\bar{A}_\tau)$, which are either zero or one, are in $S^\prime$. Inequality \eqref{Eq:TwoComLapBound} implies that eigenvalues of $L(A_\tau)$ are at most $\d$ away from the corresponding eigenvalues of $L(\bar{A}_\tau)$. Therefore the second largest eigenvalue of $L(A_\tau)$ is in $S$ and the remaining eigenvalues of $L(A_\tau)$ are in $S^\prime$.

Note that $S$ and $S^\prime$ are disjoint because $\d$ is small compared to $\l_2$. In fact, from the definition of $\d$, assumption \eqref{eq: Gnab condition}
(increasing the constant $C$ if necessary), and \eqref{Eq:EigValBound} we have
\begin{equation}\label{eq: bound of delta}
  \d \leq \frac{3C^\prime}{\sqrt{a}}\cdot\frac{\e(a-b)}{\sqrt{C(a+b)}}
  \le\frac{3C^\prime\e}{\sqrt{C}}\cdot\frac{a-b}{a+b}
  \le \frac{\e \l_2}{20}.
\end{equation}
Using \eqref{Eq:EigValBound} and \eqref{eq: bound of delta}, we bound the distance $\text{dist}(S,S^\prime)$ between $S$ and $S^\prime$ as follows:
\begin{eqnarray}\label{Eq:SpecGap}
  \text{dist}(S,S^\prime)=\min\Big\{\lambda_2-2\d,\frac{1}{5}-2\d\Big\}
  \ge \frac{\lambda_2}{4}-2\d > \frac{\lambda_2}{8}.
\end{eqnarray}
Applying Theorem VII.3.2 in \cite{Bhatia1996} and using \eqref{Eq:TwoComLapBound}, \eqref{Eq:SpecGap}, \eqref{eq: bound of delta} we obtain
\begin{equation}\label{eq: projection concentration}
  \left\|v v^T-\bar{v}\bar{v}^T\right\|
  \le \frac{(\pi/2)\left\|L(A_\tau)-L(\bar{A}_\tau)\right\|}{\mathrm{dist}(S,S^\prime)}
  \le \frac{(\pi/2)\d}{\l_2/8}
  \le \frac{\pi\e}{5}.
\end{equation}
It is easy to check that
\begin{eqnarray}\label{Eq: eigenvector concentration}
  \min_{\beta=\pm 1}\|v+\beta\bar{v}\|_2 \leq \sqrt{2}\left\|v v^T-\bar{v}\bar{v}^T\right\|.
\end{eqnarray}
Therefore from \eqref{eq: projection concentration} and \eqref{Eq: eigenvector concentration} we have
$\min_{\beta=\pm 1}\|v+\beta\bar{v}\|_2 \le \e$. The proof is complete.
\end{proof}
